\documentclass[10pt, oneside,reqno]{amsart}
\usepackage{amssymb}
\usepackage{amsmath}
\usepackage[normalem]{ulem}
\usepackage{multirow}
\usepackage{array}
\usepackage{graphicx}
\graphicspath{ {./images/} }
\usepackage{xcolor}
\usepackage{graphicx} 
\usepackage{mathrsfs}
\usepackage[textwidth=14mm]{todonotes}
\usepackage{cancel}
\usepackage{amsthm}
\usepackage{full page}
\allowdisplaybreaks
\usepackage{geometry,mathtools}
\newtheorem{thm}{Theorem}[section]
\newtheorem{prop}[thm]{Proposition}
\newtheorem{cor}[thm]{Corollary}
\newtheorem{lem}[thm]{Lemma}
\theoremstyle{definition}

\newtheorem{rem}[thm]{Remark}
\newtheorem{dfn}[thm]{Definition}

\newtheorem{ex}[thm]{Example}

\reversemarginpar

\newcommand{\End}{\mbox{\rm End}}
\newcommand{\Hom}{\mbox{\rm Hom}}
\newcommand{\Res}{\mbox{\rm Res}}

\begin{document}
\title{On $\mathbb{N}$-graded vertex algebras associated with Gorenstein algebras}

\author{Alex Keene}
\address{Department of Mathematics, Illinois State University, Normal, IL 61790}
\email{akeene@ilstu.edu}
\author{Christian Soltermann}
\address{Department of Mathematics, Illinois State University, Normal, IL 61790} \email{cjsolte@ilstu.edu} 

\author{Gaywalee Yamskulna}\address{Department of Mathematics, Illinois State University, Normal, IL 61790} \email{gyamsku@ilstu.edu } 
\thanks{The last named author was supported by the College of Arts and Sciences at Illinois State University}

\subjclass{Primary 17B69}


\keywords{Vertex algebras, Indecomposable, Irrational, and Gorenstein algebras.}

\begin{abstract}This paper investigates the algebraic structure of indecomposable $\mathbb{N}$-graded vertex algebras $V = \bigoplus_{n=0}^{\infty} V_n$, emphasizing the intricate interactions between the commutative associative algebra $V_0$, the Leibniz algebra $V_1$ and how non-degenerate bilinear forms on $V_0$ influence their overall structure. We establish foundational properties for indecomposability and locality in $\mathbb{N}$-graded vertex algebras, with our main result demonstrating the equivalence of locality, indecomposability, and specific structural conditions on semiconformal-vertex algebras. The study of symmetric invariant bilinear forms of semiconformal-vertex algebra is investigated. We also examine the structural characteristics of $V_0$ and $V_1$, demonstrating conditions under which certain $\mathbb{N}$-graded vertex algebras cannot be quasi vertex operator algebras, semiconformal-vertex algebras, or vertex operator algebras, and explore $\mathbb{N}$-graded vertex algebras $V=\bigoplus_{n=0}^{\infty}V_n$ associated with Gorenstein algebras. Our analysis includes examining the socle, Poincar\'{e} duality properties, and invariant bilinear forms of $V_0$ and their influence on $V_1$, providing conditions for embedding rank-one Heisenberg vertex operator algebras within $V$. Supporting examples and detailed theoretical insights further illustrate these algebraic structures.
\end{abstract}
\maketitle
\section{Introduction}
The majority of research on $\mathbb{N}$-graded vertex algebras $V = \bigoplus_{n=0}^{\infty} V_n$ and their representations—both from mathematical and physical perspectives—has often concentrated on vertex algebras of CFT type, which are characterized by having $V_0 = \mathbb{C}{\bf 1}$. However, other substantial classes of $\mathbb{N}$-graded vertex algebras fall outside this scope. Examples include vertex algebras derived from Toroidal Lie algebras (see \cite{Berman2002Representations, Li2012Toroidal, Kong2016Twisted, Szczesny2021Toroidal}, and reference therein) and the vertex algebra associated with the $\beta\gamma$-system, which is fundamental in free field realizations of affine Lie algebras and has played a pivotal role in multiple aspects of conformal field theory (see for instance \cite{Wakimoto1986Fock, Frenkel2001Vertex, Feigin1988Family, Feigin1990Affine, Feigin1990Representations, Barron2022Rationality} and references therein).

Over time, numerous developments have underscored the critical role that the representation theory of $\mathbb{N}$-graded vertex algebras $V=\bigoplus_{n=0}^{\infty}V_n$ such that $\dim V_0\geq 2$ has played in advancing the general theory of vertex algebras. In a series of studies involving Gerbes of chiral operators and the chiral de Rham complex \cite{MSV, MS1, MS2}, the structure of $\mathbb{N}$-graded vertex algebras $V = \bigoplus_{n=0}^{\infty} V_n$ with dimension of $V_0$ not necessary one dimensional has been investigated. For such algebras, the bilinear operation $(u, v) \mapsto u_i v$ (for $i \in\{ 0,1\}$) is closed on $V_0 \oplus V_1$, with the skew-symmetry and Jacobi identities giving rise to compatibility relations that form the foundation of the notion of a 1-truncated conformal algebra. Furthermore, $V_0$ with the product $(a, a') \mapsto a_{-1} a'$ is a unital commutative associative algebra with ${\bf 1}$ as the identity, while $V_0$ (as a nonassociative algebra) acts on $V_1$ by $a \cdot u = a_{-1} u$. The algebraic structure of $V_0 \oplus V_1$ is captured by the concept of a vertex $\mathcal{A}$-algebroid, where $\mathcal{A}$ is a commutative associative algebra. In frontier work by Gorbounov, Malikov, Schechtman, and Vaintrob, these $\mathbb{N}$-graded vertex algebras were linked to the cohomology of the chiral de Rham complex on a complex manifold $M$, identifying $V_0$ with the de Rham cohomology $H^*(M)$. The construction of an $\mathbb{N}$-graded vertex algebra $V = \bigoplus_{n=0}^{\infty} V_n$ from any vertex $A$-algebroid, ensuring $V_0 = \mathcal{A}$ and with the vertex $\mathcal{A}$-algebroid $V_1$ isomorphic to the original vertex algebroid was established. Li and the third author of this paper further classified simple twisted and non-twisted modules of the vertex algebras linked to vertex algebroids in terms of simple modules for specific Lie algebroids \cite{LiY, LiY2}. Recently, Barnes, Martin, Service, and the third author of this paper constructed and analyzed $\mathbb{N}$-graded vertex algebras associated with vertex algebroids that are cyclic non-Lie left Leibniz algebras of small dimensions, showing relationships between these vertex algebras and the vertex operator algebra associated with a rank-one Heisenberg algebra \cite{Barnes2024}.

The $C_2$-cofiniteness property has been vital in studying vertex algebra representation theory. Although much literature on vertex algebras and their representations has focused on rational $C_2$-cofinite vertex algebras \cite{Borcherds1986Vertex, Borcherds1992Monstrous, Dong1993Vertex, DongLepowsky1993Generalized, Dong2000Modular, Dong2006Shifted, Dong1997Quantum, FLM2, Frenkel1992Vertex}, studies on irrational $C_2$-cofinite vertex algebras remain limited, with few families documented (see for instance \cite{Abe2007Orbifold, Adamovic2014C2Cofinite, Adamovic2019Fusion, Carqueville2006Nonmeromorphic, Ridout2015Bosonic} and reference therein). Identifying additional families of irrational $C_2$-cofinite vertex algebras thus remains essential. In \cite{JY2, BuY}, Jitjankarn, Bui, and the third author constructed indecomposable, non-simple $\mathbb{N}$-graded vertex algebras meeting the $C_2$-cofiniteness condition by employing finite-dimensional vertex $\mathcal{A}$-algebroids $\mathcal{B}$ that are simple Leibniz algebras. The one-to-one correspondence between simple modules of these vertex algebras and specific types of rational affine vertex operator algebras were established.

Theoretical advancements have also shown that an $\mathbb{N}$-graded vertex operator algebra $V = \bigoplus_{n=0}^{\infty} V_n$ with $\dim V_0 \geq 2$ is indecomposable if and only if $V$ is local if and only if $V_0$ is a local algebra, as demonstrated in \cite{DongMason2004}. Later, Jitjankarn and the third author formulated criteria for $\mathbb{N}$-graded vertex algebras with $\dim V_0 \geq 2$ to be indecomposable non-simple vertex algebras, exploring the influence of simple Leibniz algebras on their structure (cf. \cite{JY}). In \cite{Mason2014}, Mason and the third author examined self-dual $\mathbb{N}$-graded vertex operator algebras that are $C_2$-cofinite, demonstrating how vertex operators within the Levi factor of the Leibniz algebra $V_1$ generate an affine Kac-Moody vertex operator subalgebra and how, in cases with a de Rham structure, $V_0$ exhibits characteristics similar to the de Rham cohomology of a complex manifold with Poincaré duality.

This paper aims to explore further the algebraic structure of indecomposable $\mathbb{N}$-graded vertex algebras $V = \bigoplus_{n=0}^{\infty} V_n$ and the roles of the commutative associative algebra $V_0$ and the Leibniz algebra $V_1$ within $V$. We investigate how non-degenerate bilinear forms on $V_0$ affect the structure of $\mathbb{N}$-graded vertex algebras. Section \ref{indecomposablesection} sets the foundation for studying vertex algebras' indecomposability and locality properties. The main result, Theorem \ref{Maintheoremforindecomposibility}, states that if $V = \bigoplus_{n=0}^{\infty} V_n$ is a semiconformal-vertex algebra with a countable number of maximal ideals, then following conditions are equivalent:
$V$ is local;
$V$ is indecomposable;
$Z(V) = \{ v \in V \mid D(v) = 0 \}$ is a local algebra; and 
$V_0$ is a local algebra.
This theorem generalizes the result in \cite{DongMason2004} for $\mathbb{N}$-graded vertex operator algebras. In Section \ref{bilinearformssection}, we extend the theory of symmetric invariant bilinear forms on vertex operator algebras that was developed by Li in \cite{Li1994} to semiconformal-vertex algebras, establishing relationships with the dual space of $V_0 / L(1)V_1$. This result fills gaps in the literature by providing a rigorous approach to understanding symmetric invariant bilinear forms for semiconformal-vertex algebras, building upon and refining methods from \cite{Li1994}.
Section \ref{knownresultsonalgbraicstructuresection} reviews known results on the algebraic structures of $V_0$ and $V_1$ for certain $\mathbb{N}$-graded vertex algebras. Also, we show that if $V_1$ contains $sl_2$ and satisfies certain conditions, $V$ cannot be a quasi vertex operator algebra, semiconformal-vertex algebras, or vertex operator algebra. In Section \ref{vertexalgebrasgorensteinring}, we dive into the study of vertex algebras associated with Gorenstein algebras. We show that if $V=\bigoplus_{n=0}^{\infty}V_n$ is a $\mathbb{N}$-vertex algebra and $V_1$ contains a simple Lie algebra with a specific condition, then $V_0$ is not Gorenstein. Also, when $V_0$ is a Gorenstein algebra and $V$ satisfies suitable conditions, $V_1$ is a solvable Leibniz algebra. Later, we explore the algebraic structure of the $\mathbb{N}$-graded vertex algebras $V=\bigoplus_{n=0}^{\infty}V_n$ when $V_0$ is a Gorenstein algebra, we study the role of the socle of $V_0$, Poincar\'{e} duality property of $V_0$ and non-degenerate invariant bilinear form of $V_0$ on the algebraic structure of $V_1$ and bilinear forms on $V_1$, and we investigate conditions on $V_0\oplus V_1$ that yields a rank one Heisenberg vertex operator algebra inside $V$. Section \ref{examples} provides supporting examples. Lastly, Section \ref{appendix} includes background on the Nilradical and Jacobson radical of rings, Gorenstein algebra properties, and left Leibniz algebras.


\section{On Indecomposable semi-conformal $\mathbb{N}$-graded vertex algebras}\label{indecomposablesection}

In this section, we review the definitions of a quasi vertex operator algebra and a semiconformal-vertex algebra, along with key properties of vertex algebras that will be utilized throughout this section. Additionally, we establish several results necessary for demonstrating the equivalence between the indecomposability property, the locality property of an $\mathbb{N}$-graded vertex algebra $V = \bigoplus_{n=0}^{\infty} V_n$, and the locality property of the unital commutative associative algebra $V_0$ when $V$ is a semiconformal-vertex operator algebra. Furthermore, we examine the properties of a particular class of homomorphisms of vertex algebras that will be relevant for our discussion in Section \ref{bilinearformssection}.
\begin{dfn}\cite{LepowskyLi2004} A {\em vertex algebra} is a vector space $V$ equipped with a linear map 
\begin{eqnarray*}
    Y:V&\rightarrow&\End(V)[[x,x^{-1}]]\\
    v&\mapsto&Y(v,x)=\sum_{n\in\mathbb{Z}}v_nx^{-n-1} (\text{ where }v_n\in \End(V))
\end{eqnarray*}
where $Y(v,x)$ is called the vertex operator associated with $v$ and a distinguished vector ${\bf 1}\in V_0$ (the vacuum vector), satisfying the following conditions for $u,v\in V$:
\begin{eqnarray}
    &&u_nv=0\text{ for $n$ sufficiently large};\\
    &&Y({\bf 1},x)=1_V;\\
    &&Y(v,x){\bf 1}\in V[[x]]\text{ and }\lim_{x\rightarrow 0}Y(v,x){\bf 1}=v;\\
    &&x_0^{-1}\delta\left(\frac{x_1-x_2}{x_0}\right)Y(u,x_1)Y(v,x_2)-x_0^{-1}\delta\left(\frac{x_2-x_1}{-x_0}\right)Y(v,x_2)Y(u,x_1)\\
    &&=x_2^{-1}\delta\left(\frac{x_1-x_0}{x_2}\right) Y(Y(u,x_0)v,x_2).\nonumber
\end{eqnarray} 
\end{dfn}
We define a linear operator $D$ on $V$ by $D(v)=v_{-2}{\bf 1}$ for $v\in V$. Then for $v\in V$, $n\in\mathbb{Z}$, $Y(v,x){\bf 1}=e^{xD}v$, ${[D,v_n]}=(Dv)_n=-nv_{n-1}$.

\begin{prop}\cite{LepowskyLi2004} Let $V$ be a vertex algebra. Then $Y(u,x)v=e^{xD}Y(v,-x)u$. 
\end{prop}

\begin{dfn}\cite{LepowskyLi2004} An {\em ideal} of the vertex algebra $V$ is a subspace $I$ of $V$ such that $v_nw, w_nv\in I$ for all $v\in V, w\in I, n\in\mathbb{Z}$.
\end{dfn}
\begin{rem}
    Under the condition that $D(I)\subseteq I$, the left ideal condition, $v_nw\in I$ for $v\in V, w\in I, n\in\mathbb{Z}$ is equivalent to the right-ideal condition $w_nv\in I$ for $v\in V, w\in I, n\in\mathbb{Z}$.
\end{rem}

We define $Z(V)$, the center of $V$, to be the set $\{v\in V~|~Dv=0\}$.


\begin{lem}\ \

\begin{enumerate}
    \item Let $v\in V$. The following statements are equivalent:
    \begin{enumerate}
        \item $v\in Z(V)$;
    \item The vertex operator for $v$ is a constant, i.e. $Y(v,x)=v_{-1}$.
    \end{enumerate}
    \item $Z(V)$ is a unital commutative associative algebra with respect to the $_{-1}$-product $a*b=a_{-1}b$ for $a,b\in Z(V)$.
    \item For $v\in Z(V)$, we define the annihilator of $v$ as $Ann_V(v)=\{u\in V~|~v_{-1}u=0\}.$ Then $Ann_V(v)$ is an ideal of $V$.
\end{enumerate}
    
\end{lem}
\begin{proof} The proofs of the statements (1), and (2) are exactly the same as the proof of Lemma 2.1, Lemma 2.3 in \cite{DongMason2004}. We only need to replace $L(-1)$ in the proofs of these Lemmas by $D$. 

Now we will prove statement (3): $Ann_V(v)$ is an ideal of $V$. Let $a\in Ann_V(v)$ and $u\in V$, $n\in\mathbb{Z}$. Because
$v_{-1}u_na=u_nv_{-1}a+\sum_{i=0}^{\infty}(-1)^i(v_iu)_{-1+n-i}a=0$, and
    $v_{-1}D(a)=Dv_{-1}a-(Dv)_{-1}a=0$,
 we can conclude that $Ann_V(v)$ is an ideal of $V$.
\end{proof}
\begin{cor}
    Let $e\in Z(V)$. Then ${[e_{-1},v_n]=0}$ for all $v\in V$, $n\in\mathbb{Z}$. In particular, we have $e_{-1}v=v_{-1}e$ for all $v\in V$.
\end{cor}
\begin{proof}
    Since $e_m=0$ for all $m\neq -1$, $n\in\mathbb{Z}$, we have $$[e_{-1},v_n]=\sum_{i=0}^{\infty}\binom{n}{i}(e_iv)_{-1+n-i}=0.$$ This implies that 
    $e_{-1}v=e_{-1}v_{-1}{\bf 1}=v_{-1}e_{-1}{\bf 1}=v_{-1}e$.
\end{proof}

\begin{prop}\label{Z(V)}\cite{LepowskyLi2004} 
Let $(V,Y,{\bf 1})$ be a vertex algebra. Let $v\in V$ such that $D(v)=0$. Then  $Y(u,x)v=e^{xD}u_{-1}v$ for $u\in V.$  
\end{prop}


We define $$\Hom_{V,D}(V,V):=\{f\in \End_{\mathbb{C}}(V,V)~|~fY(u,x)=Y(u,x)f,~fD=Df\text{ for all }u\in V\}.$$ Clearly, $\Hom_{V,D}(V,V)$ is a ring. Let $e\in Z(V)$. Let $f_e:V\rightarrow V$ be a linear map defined by $f_e(v)=e_{-1}v$. Notice that
\begin{equation}\label{imageoffe}
f_e(u_nv)=e_{-1}u_nv=u_ne_{-1}v=u_nf_e(v)\end{equation} for all $u,v\in V$, $n\in\mathbb{Z}$. In addition, $f_e(Dv)=e_{-1}Dv=De_{-1}v=Df_e(v)$ for all $v\in V$. Hence, $f_e\in \Hom_V(V,V)$.

Because of (\ref{imageoffe}), and $Df_e(v)=De_{-1}v=e_{-1}D(v)+e_{-2}v=e_{-1}D(v)$, the space $f_e(V)$ is an ideal of $V$. Similarly, since ${\bf 1}-e\in Z(V)$, we can conclude that $f_{{\bf 1}-e}(V)$ is an ideal of $V$.

Now, let $\varphi\in Hom_{V,D}(V,V)$. Then $\varphi({\bf 1})\in Z(V)$ because $D(\varphi({\bf 1}))=\varphi(D({\bf 1}))=0$. Since $f_{\varphi({\bf 1})}(v)=\varphi({\bf 1})_{-1}v=v_{-1}\varphi({\bf 1})=\varphi(v_{-1}{\bf 1})=\varphi(v)$ for all $v\in V$, we can conclude that $\varphi=f_{\varphi({\bf 1})}$.
\begin{prop}\label{Hom(VV)} Let $V$ be a vertex algebra $V$. Let $f:Z(V)\rightarrow \Hom_{V,D}(V,V)$ be a linear map defined by $f(e)=f_e$. Then the linear map $f$ is a ring isomorphism.
\end{prop}
\begin{proof} Let $f:Z(V)\rightarrow \Hom_{V,D}(V,V)$ be a linear map defined by $f(e)=f_e$. Clearly, $f$ is onto. Now, we will show that $f$ is one-to-one. Assume that $f(a)=0$. Hence $a_{-1}v=f_a(v)=0$ for all $v\in V$. In particular, $a=a_{-1}{\bf 1}=0$. Therefore, $Ker(f)=\{0\}$.

Next,let $a,a'\in Z(V)$. We have
\begin{eqnarray*}
   f_{a_{-1}a'}(v)&=&(a_{-1}a')_{-1}v\\
   &=&\sum_{i=0}^{\infty}a_{-1-i}a'_{-1+i}v+a'_{-2-i}a_iv\\
   &=&a_{-1}a'_{-1}v\\
   &=&f_a(a'_{-1}v)\\
   &=&f_a(f_{a'}(v))\\
   &=&f_a\circ f_{a'}(v)\text{ for all }v\in V.
\end{eqnarray*} So, $f$ is a ring isomorphism. \end{proof}

Recall that a commutative ring $R$ is local if it has a unique maximal ideal. Moreover, if the ring $R$ is local then it contains no idempotents except $0$ and $1_R$. Now, we will investigate relationships between indecomposable properties of vertex algebras $V$ and local properties of $Z(V)$.

\begin{lem}\label{indecomposableidempotent}
    Let $V$ be a vertex algebra. Then the following equivalent
    \begin{enumerate}
        \item $V$ is indecomposable.
        \item The only idempotents of $Z(V)$ are $ 0, {\bf 1} $.
    \end{enumerate}
In addition if $Z(V)$ is local then $V$ is indecomposable.
\end{lem}
\begin{proof} First, let us assume that $Z(V)$ contains an idempotent $e$ such that $e\neq 0, {\bf 1} $. Clearly, $f_{e}(V)+f_{{\bf 1}-e}(V)\subset V $. Observe that for $v\in V$, we have $v=e_{-1}v+({\bf 1}-e)_{-1}v\in f_{e}(V)+f_{{\bf 1}-e}(V).$ Hence, $V=f_{e}(V)+f_{{\bf 1}-e}(V)$. Let $u\in f_{e}(V)\cap f_{{\bf 1}-e}(V)$. Then 
$u=e_{-1}v^1=({\bf 1}-e)_{-1}v^2$
for some $v^1,v^2\in V$. 
Using the fact that ${\bf 1}-e$, $e$ are idempotents in $Z(V)$, and iterate formula, we have $0=(({\bf 1}-e)_{-1}e)_{-1}v^1=({\bf 1}-e)_{-1}e_{-1}v^1=({\bf 1}-e)_{-1}v^2=u$. Therefore, $V=f_{e}(V)\oplus f_{{\bf 1}-e}(V)$ and $V$ is decomposable.

Next, we assume that $V$ is decomposable. Hence, there exist ideals $V^1 $ and $V^2$ such that $V=V^1\oplus V^2$. Let $\iota_i:V^i\rightarrow V$ and $p_j:V\rightarrow V^j$ be module homomorphisms. Now, we set $e^j=\iota_j\circ p_j$. Observe that 
$e^i\circ e^i=(\iota_i\circ p_i)\circ (\iota_i\circ p_i)=e^i$ and $e^i\circ e^j=(\iota_i\circ p_i)\circ (\iota_j\circ p_j)=0$ when $i\neq j$. In addition, $e_1+e_2=id_V$ where $id_V$ is the identity map. Therefore, $e_1,e_2$ are idempotents in $\Hom_{V,D}(V,V)$ such that $e_i\neq 0,~id_V$ for $i\in\{1,2\}$. Since $Z(V)$ and $\Hom_{V,D}(V,V)$ are isomorphic as rings, we can conclude that $Z(V)$ contains an idempotent $e$ such that $e\neq 0,{\bf 1}$. Therefore, $V$ is indecomposable if and only if $Z(V)$ contains no idempotent $\neq 0,{\bf 1}$. The rest is clear.
\end{proof} 

Recall that a finite-dimensional non-associative unital algebra $\mathcal{A}$ is called power associative if the subalgebra generated by any element in $\mathcal{A}$ is associative. 
\begin{prop}\label{powerassociative}\cite{DongMason2004} Let $\mathcal{A}$ be a finite-dimensional unital power associative algebra. If the identity $1_{\mathcal{A}}$ is the only nonzero idempotent in $\mathcal{A}$, then $\mathcal{A}$ is local.  
\end{prop}

\begin{thm}\label{VZ(V)} Let $V$ be a vertex algebra such that $\dim Z(V)<\infty$. Then the following are equivalent
\begin{enumerate}
    \item $V$ is indecomposable.
    \item The only idempotents of $Z(V)$ are $0$ and ${\bf 1}$.
    \item $Z(V)$ is a local algebra. 
\end{enumerate}
\end{thm}
\begin{proof}
    It follows immediately from Lemma \ref{indecomposableidempotent} and Proposition \ref{powerassociative}.
\end{proof}

\begin{dfn}\cite{FrenkelHuangLepowsky} A {\em quasi} vertex operator algebra is a vertex algebra $V=\bigoplus_{n\in\mathbb{Z}}V_n$ that satisfies the following conditions 
\begin{itemize}
    \item for $n\in\mathbb{Z}$, $\dim V_n<\infty$;
    \item $V_n=0$ for $n$ sufficiently small; 
    \item there is a representation $\rho$ of $ \mathfrak{sl}(2)$ on $V$ given by: $L(j)=\rho(L_j)$, $j\in\{0,\pm 1\}$, where $\{L_{-1},L_0,L_1\}$ is a basis of $\mathfrak{sl}(2)$ with the Lie brackets
$$[L_0,L_{-1}]=L_{-1},~[L_0,L_1]=-L_1\text{ and }[L_{-1}, L_1]=-2L_0,$$ 
and the following conditions hold for $v\in V$ and $j\in\{0,\pm 1\}$:
\begin{eqnarray*}
    [L(j),Y(v,x)]&=&\sum_{k=0}^{j+1}\binom{j+1}{k}x^{j+1-k}Y(L(k-1)v,z),\\
    \frac{d}{dz}Y(v,x)&=&Y(L(-1)v,x),\text{ and }\\
    L(0)v&=&nv=(wt v)v\text{ for }n\in\mathbb{Z}\text{ and }v\in V_n.
\end{eqnarray*}
\end{itemize}
\end{dfn}
\begin{rem} For $v\in V$, $L(-1)v=v_{-2}{\bf 1}=D(v)$.
\end{rem}
The Virasoro algebra $Vir$ is a Lie algebra with a basis $\{L_n~|~n\in\mathbb{Z}\}\cup\{c\}$ where 
$$[L_m,L_n]=(m-n)L_{m+n}+\frac{1}{12}(m^3-m)\delta_{m+n,0}c$$ for $m,n\in\mathbb{Z}$ and $c$ is central. We denote by $L(n)$ the operator on any $Vir$-module corresponding to $L_n$ for $n\in\mathbb{Z}$.

We set $Vir^+=span\{L_m~|~m\geq -1\}$. This is a subalgebra of the Virasoro algebra $Vir$. 
\begin{dfn}\cite{Li2019} A vertex algebra $(V=\oplus_{n\in\mathbb{Z}}V_n,Y)$ is a {\em semiconformal-vertex algebra} if it is a quasi-vertex operator algebra that satisfies the following additional conditions
\begin{itemize}
    \item there is a representation $\rho$ of $Vir^+$ on $V$ given by $L(j)=\rho(L_j)$ for $j\geq -1$,
    \item 
    $[L(j),Y(v,x)]=\sum_{k=0}^{j+1}\binom{j+1}{k}x^{j+1-k}Y(L(k-1)v,x)$ for $j\geq -1$.
\end{itemize}    
\end{dfn}

\begin{prop}\label{ZVV0} Let $V=\bigoplus_{n\in\mathbb{Z}}V_n$ be a semiconformal- vertex algebra. For $v\in Z(V)$, we have $L(m)v=0$ for all $m\geq -1$. Moreover, $Z(V)\subseteq V_0$ consists of primary states.
\end{prop}
\begin{proof}The proof of this proposition is exactly the same as the proof of Lemma 2.2 of \cite{DongMason2004}.\end{proof}
\begin{prop}\label{V_0localindecomposable}Let $V$ be a vertex algebra. If $V$ is local then $V$ is indecomposable.

\end{prop}
\begin{proof} Assume that $V$ is a local vertex algebra with the unique maximal ideal $\mathcal{M}$. If $V$ is decomposable then there exist proper ideals $M$, $N$ such that $V=M\oplus N$. Since $M\subseteq \mathcal{M}$ and $N\subseteq \mathcal{M}$, it implies that $V\subseteq \mathcal{M}$. This is impossible.\end{proof}
\begin{prop}\label{V_0localindecomposable1} \cite{JY} Let $V=\bigoplus_{n=0}^{\infty}V_n$ be an $\mathbb{N}$-graded vertex algebra such that $\dim V_0\geq 2$. If $V_0$ is a local algebra then $V$ is indecomposable.
\end{prop}

The following Theorems generalizes Theorem 2 in \cite{DongMason2004} for the case when $V$ is an $\mathbb{N}$-graded vertex operator algebra to the case when $V$ is an $\mathbb{N}$-graded semiconformal-vertex algebra.

\begin{thm}\label{indecomposablelocalV_0} Let $V=\bigoplus_{n=0}^{\infty}V_n$ be a semiconformal-vertex algebra such that $\dim V_0\geq 2$. Then the following are equivalent 
\begin{enumerate}
    \item $V$ is indecomposable.
    \item $Z(V)$ is a local algebra.
    \item $V_0$ is a local algebra.
\end{enumerate} 
\end{thm}
\begin{proof} By Theorem \ref{VZ(V)}, we can conclude immediately that (1)$\Leftrightarrow$(2). Next, we will prove (2)$\Rightarrow$(3). First, by following the proof of Proposition 4.3 of \cite{DongMason2004}, one can show that $Z(V)$ contains every idempotent of $V_0$. This result implies that if $Z(V)$ is a local algebra then $V_0$ is a local algebra. Next, by Proposition \ref{V_0localindecomposable1}, we have that (3)$\Rightarrow$ (1). This completes the proof of this theorem. \end{proof}
\begin{prop}\label{localindecomposable} Let $V=\bigoplus_{n=0}^{\infty}V_n$ be an $\mathbb{N}$-graded vertex algebra. Let $J(V)=\cap M$ be the intersection of all maximal ideals $M$ of $V$. 
\begin{enumerate}
    \item If $I=\bigoplus_{n=0}^{\infty}I_n$ is an ideal in $V$ such that $I_0=\{0\}$, then $I\subseteq J(V)$. Here $I_n=I\cap V_n$ for all $n\in\mathbb{N}$.
    \item Assume that $V$ has countably many maximal ideals and $\dim V_0<\infty$. Then $V/J(V)$ is semisimple.
\end{enumerate}

\end{prop}
\begin{proof} 
Observe that $J(V)=\bigoplus_{n=0}^{\infty}J(V)_n$ is $\mathbb{N}$-graded. Here, $J(V)_n=J(V)\cap V_n$. If $I$ is not a subset of $J(V)$, then there exists a maximal ideal $M$ of $V$ such that $I$ is not a subset of $M$. So, $M+I$ must be $V$. This implies that ${\bf 1}\in M_0+I_0=M_0$ which is impossible. So, $I\subseteq J(V)$. This proves (1).
    
    Now, we assume that $V$ has countably many maximal ideals and $\dim V_0<\infty$. Let $\mathcal{M}=\{M^1,M^2,....\}$ be collection of all maximal ideals. For each positive integer $k$, we set $\mathcal{M}^k=\cap_{i=1}^kM^i$. Notice that $\mathcal{M}^k=\bigoplus_{n=0}^{\infty}\mathcal{M}^k_n$ is an $\mathbb{N}$-graded ideal such that $\mathcal{M}^k_n=V_n\cap \mathcal{M}^k$. Moreover, we have the following relation
$$J(V)_0=\cap_{i=0}^{\infty}M^i_0\subseteq \cdots\subseteq \mathcal{M}^{k+1}_0\subseteq \mathcal{M}^{k}_0\cdots \subseteq\mathcal{M}^1_0\subseteq V_0 .$$ Since $\dim V_0<\infty$, it implies that there exists a positive integer $t_0$ such that $J(V)_0=\mathcal{M}^t=\cap_{i=0}^{t}M^i_0$ for all $t\geq t_0$. This implies that $\mathcal{M}^t/J(V)$ is an ideal such that $(\mathcal{M}^t/J(V))_0=\{0+J(V)\}$ where $(\mathcal{M}^t/J(V))_0=\mathcal{M}^t_0/J(V)$. By (1), we can conclude that $\mathcal{M}^t\subseteq J(V)\subseteq \mathcal{M}^t=\cap_{i=1}^{t_0}M^i$ for all $t\geq t_0$. In addition, $V/J(V)$ is semisimple. This proves (2).\end{proof}

\begin{thm}\label{Maintheoremforindecomposibility} Let $V=\bigoplus_{n=0}^{\infty}V_n$ be a semiconformal-vertex algebra that has countably many maximal ideals. Then the following are equivalent 
\begin{enumerate}
\item $V$ is local.
    \item $V$ is indecomposable.
    \item $Z(V)$ is a local algebra.
    \item $V_0$ is a local algebra.
\end{enumerate} 
\end{thm}
\begin{proof} By Proposition \ref{V_0localindecomposable}, we know that (1)$\Rightarrow$(2). By Theorem \ref{indecomposablelocalV_0}, we have that statements (2), (3), (4) are equivalent. 

Now, we will prove that (4)$\Rightarrow$(1). Assume that $V_0$ is local. By Proposition \ref{localindecomposable}, we have that $V/J(V)$ is semisimple, and $V/J(V)$ is a finite sum of simple semiconformal-vertex algebras. In addition, there exist maximal ideals $M^1,...,M^r$ such that $V/J(V)$ is isomorphic to $\bigoplus_{i=1}^rV/M^i$ as semi-conformal vertex algebras and $J(V)=\cap_{i=1}^rM^i$. Following the proof of Theorem 2 on page 363 of \cite{DongMason2004}, we choose $e_i\in V_0$ such that the image of $e_i+J(V)$ in $V/M^i$ is its vacuum vector. Observe that $e_i+J(V)$ is an idempotent of $V_0+J(V)$ for all $i\in\{1,...,r\}$. Since $V_0$ is local, this implies that $r=1$ and $J(V)=M^1$. Therefore, $V$ is local. \end{proof}

\section{Contragredient modules and bilinear forms of Semiconformal-vertex algebras}\label{bilinearformssection}

H. Li and M. Roitman investigated invariant bilinear forms on vertex operator algebras and quasi vertex operator algebras in \cite{Li1994}, 
 and \cite{Roitman2004Invariant}, respectively. This section extends their study by examining invariant bilinear forms on semiconformal-vertex algebras. The main result of this section is presented in Theorem \ref{invariant}. Notably, this result was referenced in M. Li's paper without proof (\cite{Li2019}). For the sake of completeness, we have provided a detailed proof here. Our work generalizes and refines H. Li's work in \cite{Li1994}.

\begin{dfn}\cite{LepowskyLi2004} Let $(V,Y,{\bf 1})$ be a vertex algebra. A {\em $V$-module} is a vector space $W$ equipped with a linear map 
\begin{eqnarray*}
    Y_W:V&\rightarrow&\End(W)[[x,x^{-1}]]\\
    v&\mapsto&Y_W(v,x)=\sum_{n\in\mathbb{Z}}v_nx^{-n-1} (\text{ where }v_n\in \End(W))
\end{eqnarray*}
such that for $u,v\in V$, $w\in W$:
\begin{eqnarray}
    &&u_nw=0\text{ for $n$ sufficiently large};\\
    &&Y_W({\bf 1},z)=1_W;\\
    &&x_0^{-1}\delta\left(\frac{x_1-x_2}{x_0}\right)Y_W(u,x_1)Y_W(v,x_2)-x_0^{-1}\delta\left(\frac{x_2-x_1}{-x_0}\right)Y_W(v,x_2)Y_W(u,x_1)\\
    &&=x_2^{-1}\delta\left(\frac{x_1-x_0}{x_2}\right) Y_W(Y(u,x_0)v,x_2).\nonumber
\end{eqnarray} 
    
\end{dfn}
\begin{prop}\cite{LepowskyLi2004} 
Let $(W,Y_W)$ be a $V$-module. Then $Y_W(Dv,x)=\frac{d}{dx}Y_W(v,x)$ for $v\in V$.
    
\end{prop}

As mentioned in \cite{LepowskyLi2004} Remark 4.1.4, while the $D$-operator is a canonical operator on the vertex algebra $V$, $D$ does not act on a $V$-module so we do not, in general, have an analog of the $D$-bracket formulas. But such formulas are important enough that we shall introduce the term ``$V$-module $(W,Y_W,d)$'' to refer to a $V$-module $(W,Y_W)$ equipped with an endomorphism $d$ on $W$ such that for $v\in V$, $[d,Y_W(v,x)]=Y_W(Dv,x)=\frac{d}{dx}Y_W(v,x).$ When we have this structure, we also have the conjugation formula
$e^{x_0d}Y_W(v,x)e^{-x_0d}=Y_W(e^{x_0D}v,x)=Y_W(v,x+x_0).$


\begin{dfn}\cite{LepowskyLi2004} Let $W$ be a $V$-module. A vector $w\in W$ is a {\em vacuum-like} vector if $u_nw=0$ for $u\in V$, $n\geq 0$.
\end{dfn}

\begin{prop}\label{Z(W)}\cite{LepowskyLi2004} Let $(V,Y,{\bf 1})$ be a vertex algebra and let $(W,Y_W,d)$ be a $V$-module. Let $w\in W$ such that $dw=0$. Then  $Y(u,x)w=e^{xd}u_{-1}w$ for $u\in V.$ In particular, $w$ is a vacuum-like vector. 
\end{prop}

\begin{dfn}
    Let $V$ be a quasi vertex operator algebra. Let $(W,Y_W)$ be a module of $V$ as a vertex algebra. $W$ is a module for $V$ as a quasi vertex operator algebra if $W$ satisfies the following additional conditions:
    \begin{itemize}
    \item $W=\bigoplus_{n\in\mathbb{Q}}W_n$ is a $\mathbb{Q}$-graded vector space;
    \item for $n\in\mathbb{Q}$, $\dim W_n<\infty$;
    \item $W_n=0$ for $n$ sufficiently small; 
    \item there is a representation $\rho$ of $ \mathfrak{sl}(2)$ on $W$ given by: $L_W(j)=\rho(L_j)$, $j\in\{0,\pm 1\}$, where $\{L_{-1},L_0,L_1\}$ is a basis of $\mathfrak{sl}(2)$ with the Lie brackets
$$[L_0,L_{-1}]=L_{-1},~[L_0,L_1]=-L_1\text{ and }[L_{-1}, L_1]=-2L_0,$$ 
and the following conditions hold for $v\in V$ and $j\in\{0,\pm 1\}$:
\begin{eqnarray*}
    [L_W(j),Y_W(v,x)]&=&\sum_{k=0}^{j+1}\binom{j+1}{k}x^{j+1-k}Y_W(L(k-1)v,x),\\
    \frac{d}{dz}Y_W(v,x)&=&Y_W(L(-1)v,x),\text{ and }\\
    L_W(0)w&=&nw=(wt (w))w\text{ for }n\in\mathbb{Q}\text{ and }w\in W_n.
\end{eqnarray*}
\end{itemize}
 \end{dfn}
\begin{rem} For a module $(W,Y_W)$ of a quasi vertex operator algebra $(V,Y,{\bf 1})$,  \begin{eqnarray*}&&[L_W(-1),Y_W(v,x)]=Y_W(L(-1)v,x)=\frac{d}{dx}Y_W(v,x),\\
&&e^{x_0L_W(-1)}Y_W(u,x)e^{-x_0L_W(-1)}=Y_W(e^{x_0L(-1)}v,x)=Y_W(v,x+x_0).
    \end{eqnarray*}
\end{rem}
Now, we let $(V,Y,{\bf 1})$ be a quasi vertex operator algebra and let $(W=\oplus_{n\in\mathbb{Q}}W_n,Y_W)$ be a module for the quasi vertex operator algebra $V$. We set 
$$W'=\oplus_{n\in\mathbb{Q}}W_n^*$$ to be the graded dual space of $W$. Here, $W_n^*$ are dual spaces of the homogeneous subspaces $W_n$. Denote by $\langle\cdot,\cdot\rangle$ the pairing between $W$ and $W'$. We define the adjoint vertex operators $Y_{W'}(v,x)$ by means of the linear map
\begin{eqnarray*}
    V&\rightarrow &(\End W')[[x,x^{-1}]]\\
    v&\mapsto&Y_{W'}(v,x)=\sum_{n\in\mathbb{Z}}v_n'x^{-n-1} \text{ where }v_n'\in \End W',
\end{eqnarray*}
determined by the conditions
\begin{eqnarray*}
    &&\langle Y_{W'}(v,x)w',w\rangle=\langle w', Y_W(e^{xL(1)}(-x^{-2})^{L(0)}v,x^{-1})v\rangle,\\
    &&\langle L_{W'}(n)w',w\rangle=\langle w',L_W(n)w\rangle\end{eqnarray*}for $v\in V$, $w'\in W'$, $w\in W$, and $n\in\{0,\pm 1\}$.
\begin{prop}
    The pair $(W',Y_{W'})$ is a module for the quasi vertex operator algebra $V$. $W'$ is called contragredient.
\end{prop}
\begin{proof}
    See the proof in Theorem 5.2.1 in \cite{FrenkelHuangLepowsky}.
\end{proof}

\begin{dfn}\cite{FrenkelHuangLepowsky, Li1994} Let $(V,Y,{\bf 1})$ be a quasi vertex operator algebra. A bilinear form $(~,~)$ on $V$ is said to be invariant if it satisfies the following conditions
\begin{eqnarray}
(Y(a,x)u,v)&=&(u,Y(e^{xL(1)}(-x^{-2})^{L(0)}a,x^{-1})v)\text{ for }a,u,v\in V;\\
(L(n)u,v)&=&(u,L(-n)v)\text{ for }n\in\{0,\pm 1\}.
\end{eqnarray}
\end{dfn}


\begin{rem}\cite{FrenkelHuangLepowsky} Assume that $(~,~)$ is an invariant bilinear form on a quasi vertex operator algebra $(V,Y,{\bf 1})$. Then for $a,u,v\in V$,
\begin{enumerate}
    \item we have $(u,Y(a,x)v)=(Y(e^{xL(1)}(-x^{-2})^{L(0)}a,x^{-1})u,v)$. 
    \item Moreover, $(~,~)$ is symmetric.
\end{enumerate}
    
\end{rem} 
Let $(~,~)$ be a symmetric invariant bilinear form on a quasi vertex operator algebra $V$. Then for $u,v\in V$, we have 
\begin{eqnarray}
    (u,v)&=&\Res_x x^{-1}(Y(u,x){\bf 1},v)\label{sys1}\\
    &=&\Res_x x^{-1}({\bf 1}, Y(e^{xL(1)}(-x^{-2})^{L(0)}u,x^{-1})v).
\end{eqnarray}
Let $f$ be the linear functional on $V_0$ defined by \begin{equation}\label{sys2}f(u)=({\bf 1},u)\end{equation} for $u\in V_0$. Since $({\bf 1},L(1)V_1)=(L(-1){\bf 1},V_1)=0$, it implies that $L(1)V_1\subseteq Ker(f)$.

\begin{dfn} Let $V$ be a vertex algebra. Let $(M,Y_M)$ and $(W,Y_W)$ be $V$-modules. A linear map $\varrho:W\rightarrow M$ is a $V$-homomorphism if 
\begin{eqnarray}
    \varrho(Y_W(a,z)w)&=&Y_M(a,z)\varrho(w)\text{ for }w\in W.
\end{eqnarray}
\end{dfn} Now, we let $V$ be a quasi vertex operator algebra and we set \begin{eqnarray*}\Hom_{M,W, L(-1)}(M,W)&=&\{f\in\End_{\mathbb{C}}(M,W)~|~f\text{ is a $V$-homomorphism and }\\
&&\hspace{4cm} fL_M(-1)=L_W(-1)f\}.
\end{eqnarray*}
\begin{prop} Let $V$ be a quasi vertex operator algebra. Let $M$ be a module of the quasi vertex operator algebra $V$. We set $Z(M)=\{w\in M~|~L_M(-1)w=0\}$.   Let $\pi:\Hom_{V,M, L(-1)}(V,M)\rightarrow Z(M);\varphi\mapsto \varphi({\bf 1})$. Then the linear map $\pi$ is an isomorphism. 
\end{prop}

\begin{proof} The proof of this proposition is very similar to the proof in Proposition 3.4 of \cite{Li1994}. For $\varphi\in \Hom_{V,M,L(-1)}(V,M)$, we have $\varphi(a)=\varphi(a_{-1}{\bf 1})=a_{-1}\varphi({\bf 1})$ for all $a\in V$. If $\pi(\varphi)=0$, then $\varphi=0$. Hence, $\pi$ is an injective map. Next, we will show that the linear map $\pi$ is an isomorphism. For $u\in Z(M)$, we define a linear map $\varphi_u:V\rightarrow M$ by $\varphi_u(a)=a_{-1}u$ for $a\in V$. Since $u\in Z(M)$, we have $Y_M(a,x)u=e^{xL_M(-1)}a_{-1}u$ (cf. Proposition \ref{Z(W)}). Because 
    \begin{eqnarray*}
        &&Y_M(Y(a,x_0)b,x_2)u\\
        &=&Y_M(a,x_0+x_2)Y_M(b,x_2)u+Y_M(b,x_2)(Y_M(a,x_0+x_2)-Y_M(a,x_2+x_0))u\\
        &=&Y_M(a,x_0+x_2)e^{x_2L_M(-1)}b_{-1}u+Y_M(b,x_2)(e^{(x_0+x_2)L_M(-1)}a_{-1}-e^{(x_2+x_0)L_M(-1)}a_{-1})u\\
        &=&Y_M(a,x_0+x_2)e^{x_2L_M(-1)}b_{-1}u,
    \end{eqnarray*}
we then have 
\begin{eqnarray*}
\varphi_u(Y(a,x_0)b)&=&\Res_{x_2}x_2^{-1}Y_M(Y(a,x_0)b,x_2)u\\
&=&\Res_{x_2}x_2^{-1}Y_M(a,x_0+x_2)e^{x_2L_M(-1)}b_{-1}u\\
&=&Y_M(a,x_0)b_{-1}u\\
&=&Y_M(a,x_0)\varphi_u(b).
\end{eqnarray*} Also, 
\begin{eqnarray*}
\varphi_u(L(-1)a)&=&(L(-1)a)_{-1}u\\
&=&L_M(-1)a_{-1}u-a_{-1}L_M(-1)u\\
&=&L_M(-1)a_{-1}u\\
&=&L_M(-1)\varphi_u(a)\end{eqnarray*} for all $a\in V$. Hence $\varphi_u\in \Hom_{V,M,L(-1)}(V,M)$ and $\pi$ is an isomorphism.
\end{proof}
\begin{dfn}
  Let $V$ be a semiconformal-vertex algebra. Let $(W=\bigoplus_{n\in\mathbb{Q}}W_n, Y)$ be a module of $V$ as vertex algebra. $W$ is a module of $V$ as semiconformal-vertex algebra if it satisfies the following additional conditions
  \begin{itemize}
    \item there is a representation $\rho$ of $Vir^+$ on $W$ given by $L_W(j)=\rho(L_j)$ for $j\geq -1$,
    \item 
    $[L_W(j),Y_W(v,x)]=\sum_{k=0}^{j+1}\binom{j+1}{k}x^{j+1-k}Y_W(L(k-1)v,x)$ for $j\geq -1$, and 
\item $W$ is a module of $V$ as a quasi vertex operator algebra.
\end{itemize}
\end{dfn}
\begin{lem}\label{CenterW}
    Let $V=\bigoplus_{n\in\mathbb{Z}}V_n$ be a semiconformal-vertex algebra and let $W=\bigoplus_{n\in\mathbb{Q}}W_n$ be a module of $V$ as a semiconformal-vertex algebra. If $u\in Z(W)$ then $L_W(m)u=0$ for all $m\geq -1$ and $u\in W_0$. 
\end{lem}
\begin{proof}
   The proof of this Lemma is exactly the same as that of Lemma 2.2 of \cite{DongMason2004}.
\end{proof}

\begin{thm}\label{invariant} Let $V$ be a semiconformal-vertex algebra. A linear functional $f$ on $V_0$ determines a symmetric invariant bilinear form by (\ref{sys1}), (\ref{sys2}) if and only if $L(1)V_1\subseteq Ker(f)$. Therefore, the space of all symmetric invariant bilinear forms on $V$ is isomorphic to the dual space of $V_0/L(1)V_1$.
\end{thm}
\begin{proof} The proof of this Theorem is almost identical to the proof of Theorem 3.1 in \cite{Li1994}. Using the fact that $\langle L_{V'}(-1)\lambda,u\rangle=\langle \lambda,L(1)u\rangle $ for any $\lambda\in V'$, $u\in V$, we can conclude that $\lambda\in Z(V')$ if and only if $\lambda(L(1)V)=0$. If $\lambda\in Z(V')$ then by Lemma \ref{CenterW}, $\lambda\in V_0'(=V_0^*)$. In fact, we can conclude further that $\lambda\in Z(V')$ if and only if $\lambda\in V_0'$ and $\lambda(L(1)V_1)=0$. 

    For any linear functional $f$ on $(V_0/L(1)V_1)^*$, we may and will consider $f$ as a member of $V'$. From the above paragraph, we can conclude that $f\in Z(V')$. Hence, we have a $V$-homomorphism $\varphi_f$ from $V$ to $V'$ such that $\varphi_f(a)=a_{-1}f$. Moreover, we have an invariant symmetric bilinear form defined by 
    $(a,b)=\langle \varphi_f(a),b\rangle$ for $a,b\in V.$ Note that $({\bf 1},b)=f(b)$ for $b\in V$. In addition, we can conclude that the space of all symmetric invariant bilinear forms on $V$ is isomorphic to the dual space of $V_0/L(1)V_1$.
\end{proof}

\section{On Algebraic Structure of $\mathbb{N}$-graded vertex algebras}\label{knownresultsonalgbraicstructuresection}

In this section, we explore the influence of vertex algebroids on $\mathbb{N}$-graded vertex algebras $V$. Specifically, we aim to identify the necessary conditions on vertex algebroids that cause $V$ to be non-semiconformal. We begin by recalling the definitions of 1-truncated conformal algebras and vertex algebroids. Subsequently, we provide the necessary background to establish Theorem \ref{Notsemisimple}, which serves as the main result of this section. This theorem asserts that for an $\mathbb{N}$-graded vertex algebra $V = \bigoplus_{n=0}^{\infty} V_n$, if $V_1$ contains a simple Lie algebra $sl_2$, then $V$ is not a quasi vertex operator algebra. Additionally, we have included a discussion on left Leibniz algebras in Section \ref{appendix} to support readers who may not be familiar with this subject.

\begin{dfn}\cite{GMS} A {\em 1-truncated conformal algebra} is a graded vector space $C=C_0\oplus C_1$ equipped with a linear map $\partial:C_0\rightarrow C_1$ and bilinear operations $(u,v)\mapsto u_iv$ for $i=0,1$ of degree $-i-1$ on $C=C_0\oplus C_1$ such that the following axioms hold:

\medskip

\noindent(Derivation) for $a\in C_0$, $u\in C_1$,
\begin{equation}
(\partial a)_0=0,\ \ (\partial a)_1=-a_0,\ \ \partial(u_0a)=u_0\partial a;
\end{equation}

\noindent(Commutativity) for $a\in C_0$, $u,v\in C_1$,
\begin{equation} 
u_0a=-a_0u,\ \ u_0v=-v_0u+\partial(u_1v),\ \ u_1v=v_1u;
\end{equation}

\noindent(Associativity) for $\alpha,\beta,\gamma\in C_0\oplus C_1$,
\begin{equation}
\alpha_0\beta_i\gamma=\beta_i\alpha_0\gamma+(\alpha_0\beta)_i\gamma.
\end{equation}
\end{dfn}
\begin{dfn} \cite{Br1, Br2, GMS} Let $(\mathcal{A},*)$ be a unital commutative associative algebra over $\mathbb{C}$ with the identity $1$. A {\em vertex $\mathcal{A}$-algebroid} is a $\mathbb{C}$-vector space $\Gamma$ equipped with 
\begin{enumerate}
\item a $\mathbb{C}$-bilinear map $\mathcal{A}\times \Gamma\rightarrow \Gamma, \ \ (a,v)\mapsto a\cdot v$ such that $1\cdot v=v$ (i.e. a nonassociative unital $\mathcal{A}$-module),
\item a structure of a left Leibniz $\mathbb{C}$-algebra $[~,~]:\Gamma\times \Gamma\rightarrow\Gamma$, 
\item a homomorphism of Leibniz $\mathbb{C}$-algebra $\pi:\Gamma\rightarrow Der(\mathcal{A})$,
\item a symmetric $\mathbb{C}$-bilinear pairing $\langle ~,~\rangle:\Gamma\otimes_{\mathbb{C}}\Gamma\rightarrow \mathcal{A}$,
\item a $\mathbb{C}$-linear map $\partial :\mathcal{A}\rightarrow \Gamma$ such that $\pi\circ \partial =0$ which satisfying the following conditions:
\begin{eqnarray*}
&&a\cdot (a'\cdot v)-(a*a')\cdot v=\pi(v)(a)\cdot \partial(a')+\pi(v)(a')\cdot \partial(a),\\
&&[u,a\cdot v]=\pi(u)(a)\cdot v+a\cdot [u,v],\\
&&[u,v]+[v,u]=\partial(\langle u,v\rangle),\\
&&\pi(a\cdot v)=a\pi(v),\\
&&\langle a\cdot u,v\rangle=a*\langle u,v\rangle-\pi(u)(\pi(v)(a)),\\
&&\pi(v)(\langle v_1,v_2\rangle)=\langle [v,v_1],v_2\rangle+\langle v_1,[v,v_2]\rangle,\\
&&\partial(a*a')=a\cdot \partial(a')+a'\cdot\partial(a),\\
&&[v,\partial(a)]=\partial(\pi(v)(a)),\\
&&\langle v,\partial(a)\rangle=\pi(v)(a)
\end{eqnarray*}
for $a,a'\in \mathcal{A}$, $u,v,v_1,v_2\in\Gamma$.
\end{enumerate}
\end{dfn}


\begin{prop}\cite{LiY} Let $(\mathcal{A},*)$ be a unital commutative associative algebra and let $\mathcal{B}$ be a module for $\mathcal{A}$ as a nonassociative algebra. Then a vertex $\mathcal{A}$-algebroid structure on $\mathcal{B}$ exactly amounts to a 1-truncated conformal algebra structure on $C=\mathcal{A}\oplus \mathcal{B}$ with 
\begin{eqnarray*}
&&a_ia'=0,\\
&&u_0v=[u,v],~u_1v=\langle u,v\rangle,\\
&&u_0a=\pi(u)(a),~ a_0u=-u_0a
\end{eqnarray*} for $a,a'\in \mathcal{A}$, $u,v\in \mathcal{B}$, $i=0,1$ such that 
\begin{eqnarray*}
&&a\cdot(a'\cdot u)-(a*a')\cdot u=(u_0a)\cdot \partial a'+(u_0a')\cdot \partial a,\\
&&u_0(a\cdot v)-a\cdot (u_0v)=(u_0a)\cdot v,\\
&&u_0(a*a')=a*(u_0a')+(u_0a)*a',\\
&&a_0(a'\cdot v)=a'*(a_0v),\\
&&(a\cdot u)_1v=a*(u_1v)-u_0v_0a,\\
&&\partial(a*a')=a\cdot \partial(a')+a'\cdot \partial(a).
\end{eqnarray*}
\end{prop}
\begin{ex}\cite{GMS}
    Let $V=\bigoplus_{n=0}^{\infty}V_n$ be an $\mathbb{N}$-graded vertex algebra. Then 
    \begin{enumerate}
        \item $V_0$ is a unital commutative associative algebra under $_{-1}$th mode.
        \item $V_0\oplus V_1$ is a 1-truncated conformal algebra.
        \item $V_1$ is a vertex $V_0$-algebroid.
    \end{enumerate}
\end{ex}


Now, we let $\mathcal{A}$ be a finite-dimensional unital commutative associative algebra with the identity $1_{\mathcal{A}}$ and let $\mathcal{B}$ be a finite-dimensional vertex $\mathcal{A}$-algebroid. By the Levi-decomposition $\mathcal{B}$ is a semidirect product of a semisimple Lie algebra $\mathcal{S}$ and $rad(\mathcal{B})$, the solvable radical of the left Leibniz algebra $\mathcal{B}$ (cf. \cite{Ba2}). Recall that $Leib(\mathcal{B})=Span\{u_0u~|~u\in \mathcal{B}\}$. Since $Leib(\mathcal{B})$ is an abelian Lie algebra, we then have that $Leib(\mathcal{B})\subseteq rad(\mathcal{B})$. Now, let us assume that $\mathcal{S}\neq 0$. Then there exist $e,f,h\in \mathcal{B}$ such that $e_0f=h$, $h_0e=2e$, $h_0f=-2f$, $h_0h=0$ and $Span\{e,f,h\}$ is a Lie algebra that is isomorphic to $sl_2$. Since $\mathbb{C}1_{\mathcal{A}}$ is an $sl_2$-module, it implies that there exists an $sl_2$-submodule $N$ of $\mathcal{A}$ such that
$\mathcal{A}=\mathbb{C}1_{\mathcal{A}}\oplus N$. 

We set $N\cdot \mathcal{B}=Span\{n\cdot b|n\in N,b\in \mathcal{B}\}$.

\begin{prop}\cite{BuY} Let $(\mathcal{A},*)$ be a finite dimensional commutative associative algebra with the identity $1_{\mathcal{A}}$ such that $\dim \mathcal{A}\geq 2$. Let $\mathcal{B}$ be a finite-dimensional vertex $\mathcal{A}$-algebroid that is not a Lie algebra. Assume that 
\begin{enumerate}
    \item there exist $e,f,h\in \mathcal{B}$ such that $e_0f=h$, $h_0e=2e$, $h_0f=-2f$, $h_0h=0$, $Span\{e,f,h\}$ is a Lie algebra that is isomorphic to $sl_2$.
    \item $\mathcal{A}=\mathbb{C}1_{\mathcal{A}}\oplus N$ and $N\cdot \mathcal{B}\subseteq \partial(\mathcal{A})$. Here $N$ is an $sl_2$-module. 
 
\begin{enumerate}
    \item If $\dim N=1$, then there exists $a\in \mathcal{A}\backslash\{0\}$ such that $a*a=0$. In addition $\mathcal{A}$ is a trivial $sl_2$-module.
    \item If $e_1f=k1_{\mathcal{A}}$ such that $k\neq 0$, then $\dim N\geq 2$.
\end{enumerate}
    \end{enumerate}
\end{prop}


\begin{prop}\label{semisimpleLeibniz}\cite{JY} Let $(\mathcal{A},*,1_{\mathcal{A}})$ be a unital commutative associative algebra such that $\dim \mathcal{A}<\infty$. Let $\mathcal{B}$ be a vertex $\mathcal{A}$-algebroid such that $\dim \mathcal{B}<\infty$. Assume that 
\begin{enumerate} 
\item $\mathcal{B}$ is semisimple Leibniz algebra such that $Leib(\mathcal{B})\neq \{0\}$, and $Ker(\partial)=\{a\in \mathcal{A}~|~u_0a=0\text{ for all }u\in \mathcal{B}\}$; 
\item the Levi-factor $\mathcal{S}=Span\{e,f,h\}$ such that $e_0f=h$, $h_0e=2e$, $h_0f=-2f$ and $e_1f=k1_{\mathcal{A}}(\neq 0)$. We set $\mathcal{A}=\mathbb{C}1_{\mathcal{A}}\bigoplus(\bigoplus_{j=1}^lN^j)$ where each $N^j$ is an irreducible $sl_2$-submodule of $\mathcal{A}$. 
\end{enumerate} Then we have the following
\begin{enumerate}
    \item $e_1e=f_1f=e_1h=f_1h=0$, $k=1$, $h_1h=21_{\mathcal{A}}$;
    \item $Ker(\partial)=\mathbb{C}1_{\mathcal{A}}$;
    \item For $j\in\{1,...,l\}$, $\dim N^j=2$ and $\dim Leib(\mathcal{B})=2l$; 
    \item $\mathcal{A}$ is a local algebra. For each $j$, we let $a_{j,0}$ be a highest weight vector of $N^j$ and $a_{j,1}=f_0(a_{j,0})$. Then $\{1_{\mathcal{A}},a_{j,i}~|~j\in\{1,...,l\}, i\in\{0,1\}\}$ is a basis of $\mathcal{A}$, and $\{\partial(a_{j,i})~|~j\in\{1,...,l\},i\in \{0,1\}\}$ is a basis of $Lieb(\mathcal{B})$.
\end{enumerate} 
Relations among $a_{j,i},e,f,h,\partial(a_{j,i})$ are described below:
\begin{eqnarray}
&&a_{j,i}*a_{j',i'}=0,\label{rel1}\\
&&a_{j,0}\cdot e=0,~a_{j,1}\cdot e=\partial(a_{j,0}),\\
&&a_{j,0}\cdot f=\partial(a_{j,1}),~a_{j,1}\cdot f=0,\\
&&a_{j,0}\cdot h=\partial(a_{j,0}),~a_{j,1}\cdot h=-\partial(a_{j,1}),\\
&&a_{j,i}\cdot \partial(a_{j',i'})=0,\label{rel5}\\
&&\partial(a_{j,i})_1e=e_0a_{j,i}=(2-i)a_{j,i-1},\\
&&\partial(a_{j,i})_1f=f_0a_{j,i}=(i+1)a_{j,i+1},\\
&&\partial(a_{j,i})_1h=h_0a_{j,i}=(1-2i)a_{j,i}.\label{rel8}
\end{eqnarray}

\end{prop}


\begin{thm}\label{Notsemisimple}
    Let $V=\bigoplus_{n=0}^{\infty}V_n$ be an $\mathbb{N}$-graded vertex algebra such that $V_1$ contains a simple Lie algebra $\mathfrak{s}$ that is isomorphic to $sl_2$. Then $V_0=\mathbb{C}{\bf 1}\oplus N$. Here, $N$ is an $sl_2$-module. Assume that 
    \begin{enumerate}
        \item there exist $u,v\in \mathfrak{s}$ such that $u_1v\neq 0$;
        \item $Ker(\partial)=\mathbb{C}{\bf 1}$;
        \item $N\cdot \mathfrak{s}=Span\{a\cdot s~|~a\in N, ~s\in\mathfrak{s}\}\subseteq \mathfrak{s}+\partial(V_0)$. Here $\partial:V_0\rightarrow V_1$ is a linear map defined by $\partial(a)=D(a)=a_{-1}{\bf 1}$.
    \end{enumerate} Then $V$ is not a quasi vertex operator algebra. 
\end{thm}
\begin{proof} Assume that $Ker(\partial)=\mathbb{C}{\bf 1}$ and there exist $u,v\in\mathfrak{s}$ such that $u_1v\neq 0$. Since $\mathfrak{s}$ is a simple Lie algebra, we then have that $\partial(b_1b')=0$ for all $b,b'\in\mathfrak{s}$. Because $Ker(\partial)=\mathbb{C}{\bf 1}$, we have $b_1b\in\mathbb{C}{\bf 1}$. Now, we let $(~,~):\mathfrak{s}\times\mathfrak{s}\rightarrow\mathbb{C}$ be a bilinear form such that for $b,b'\in \mathfrak{s}$, $(b,b'){\bf 1}=b_1b'$. Clearly, $(~,~)$ is symmetric. Observe that for $b,b',b''\in \mathfrak{s}$, $b_1(b'_0b'')=b'_0(b_1b'')-(b'_0b)_1b''=-(b'_0b)_1b''=(b_0b')_1b''$. Hence $(b,b'_0b'')=(b_0b',b'')$ and $(~,~)$ is invariant. Now, we set $U=\{w\in\mathfrak{s}~|~(w,s)=0\text{ for all }s\in\mathfrak{s}\}$. Notice that for $b\in \mathfrak{s}$, $w\in U$, we have $(b_0w,s)=-(w_0b,s)=-(w,b_0s)=0$ for all $s\in\mathfrak{s}$. Hence, $U$ is an ideal of $\mathfrak{s}$. Since $u_1v\neq 0$, this implies that $(u,v){\bf 1}=u_1v\neq 0$ and $U\neq \mathfrak{s}$. Since $\mathfrak{s}$ is a simple Lie algebra, we can conclude that $U=\{0\}$. Therefore, $(~,~)$ is a non-degenerate invariant bilinear form, and $(~,~)$ is a non-zero scalar multiple of the Killing form of $\mathfrak{s}$.

For simplicity, we suppose that $\mathfrak{s}=Span\{e,f,h\}$ such that $e_0f=h$, $h_0e=2e$, $h_0f=-2f$. Since $(~,~)$ is a nonzero scalar multiple of the Killing form of $\mathfrak{s}$, we then have that $e_1f=k{\bf 1}$ where $k\neq 0$. Since $N\cdot \mathfrak{s}\subseteq \mathfrak{s}+\partial(V_0)$, we can conclude that $\mathfrak{s}+\partial(V_0)$ is a vertex $V_0$-algebroid. Moreover, $\mathfrak{s}+\partial(V_0)$ is a semisimple Leibniz algebra. By Proposition \ref{semisimpleLeibniz}, there exists a positive integer $l$ such that $\{{\bf 1},a_{j,i}~|~j\in\{1,...,l\}, i\in\{0,1\}\}$ is a basis of $V_0$. Here, $N=\bigoplus_{j=1}^lN^j$ where each $N^j$ is an irreducible $sl_2$-module. For each $j$, $a_{j,0}$ is a highest weight vector of $N^j$ and $a_{j,1}=f_0(a_{j,0})$ with properties 
(\ref{rel1}) to (\ref{rel8}) in Proposition \ref{semisimpleLeibniz}. In addition, as commutative associative algebras, $$V_0\cong\mathbb{C}[x_i,y_j|1\leq i,j\leq l]/(x_ix_j, y_iy_j,x_iy_j~|~1\leq i,j\leq l).$$ 

Suppose that $V$ is a quasi vertex operator algebra. We set 
\begin{eqnarray*}
L(1)h&=&\alpha_0{\bf 1}+\sum_{j=1}^l (\alpha_{j,1} a_{j,1}+\alpha_{j,0} a_{j,0}),\\
L(1)e&=&\beta_0{\bf 1}+\sum_{j=1}^l(\beta_{j,1} a_{j,1}+\beta_{j,0} a_{j,0}),\\
L(1)f&=&\gamma_0{\bf 1}+\sum_{j=1}^l(\gamma_{j,1} a_{j,1}+\gamma_{j,0} a_{j,0}).
\end{eqnarray*}
Recall that for $u,v\in V_1$, we have $L(1)(u_0v)=(L(1)u)_0v+u_0L(1)v$. Because
\begin{eqnarray*}
   L(1)(h)&=& L(1)(e_0f)\\
   &=&(L(1)e)_0f+e_0L(1)f\\
  &=&(\beta_0{\bf 1}+\sum_{j=1}^l(\beta_{j,1} a_{j,1}+\beta_{j,0} a_{j,0}) )_0f+e_0( \gamma_0{\bf 1}+\sum_{j=1}^l(\gamma_{j,1} a_{j,1}+\gamma_{j,0} a_{j,0}))\\
   &=&-\sum_{j=1}^l\beta_{j,0} f_0a_{j,0}+\sum_{j=1}^l\gamma_{j,1} e_0a_{j,1}\\
   &=&-\sum_{j=1}^l\beta_{j,0} a_{j,1}+\sum_{j=1}^l\gamma_{j,1} a_{j,0},
\end{eqnarray*}
we then have that $\alpha_0=0$, $\alpha_{j,1}=-\beta_{j,0}$, $\alpha_{j,0}=\gamma_{j,1}$ for all $j\in\{1,...,l\}$. Similarly, since  
\begin{eqnarray*}
   2L(1)e&=&L(1)(h_0e)\\
   &=&(L(1)h)_0e+h_0L(1)e\\
   &=&\sum_{j=1}^l (-\beta_{j,0} (a_{j,1})_0e+\gamma_{j,1} (a_{j,0})_0e)+h_0(\beta_0{\bf 1}+\sum_{j=1}^l(\beta_{j,1} a_{j,1}+\beta_{j,0} a_{j,0}))\\
   &=&\sum_{j=1}^l (\beta_{j,0} e_0(a_{j,1})-\gamma_{j,1} e_0(a_{j,0}))+\sum_{j=1}^l(\beta_{j,1} h_0a_{j,1}+\beta_{j,0} h_0a_{j,0})\\
   &=&\sum_{j=1}^l\beta_{j,0} a_{j,0}+\sum_{j=1}^l(-\beta_{j,1} a_{j,1}+\beta_{j,0} a_{j,0})\\
    &=&-\sum_{j=1}^l\beta_{j,1} a_{j,1},
\end{eqnarray*}
we then have that $2(\beta_0{\bf 1}+\sum_{j=1}^l(\beta_{j,1} a_{j,1}+\beta_{j,0} a_{j,0}))=-\sum_{j=1}^l\beta_{j,1} a_{j,1}.$ Consequently, we have $\beta_0=0$, $\beta_{j,i}=0$ for all $1\leq j\leq l$, $i\in\{0,1\}$ and $L(1)e=0$.

Recall that for $a\in \mathcal{A},b\in \mathcal{B}$, we have $L(1)(a_{-1} b)=a_{-1}L(1)b-a_0b$.
Hence, for $1\leq j\leq l$, we have
$0=L(1)L(-1)a_{j,0}=L(1)\partial(a_{j,0})=L(1)((a_{j,1})_{-1} e)=(a_{j,1})_{-1}L(1)e-(a_{j,1})_0e=a_{j,0}$. This implies that $V_0=\mathbb{C}{\bf 1}$ which is a contradiction. Hence, $V$ is not a quasi vertex operator algebra.

\end{proof}

\section{On Indecomposable $\mathbb{N}$-graded semiconformal-vertex algebras associated with Gorenstein Algebras}\label{vertexalgebrasgorensteinring}

In this section, we delve into the intricate relationships between Gorenstein algebras, vertex algebroids, and $\mathbb{N}$-graded vertex algebras $V = \bigoplus_{n=0}^{\infty} V_n$. We begin by demonstrating that when the vertex $V_0$-algebroid $V_1$ contains a simple Lie algebra and satisfies certain specific conditions, $V_0$ cannot be a graded Gorenstein algebra (see Theorem \ref{notGorenstein}). We then show that if $V_0$ is Gorenstein and $V_1$ meets appropriate criteria, $V_1$ is a solvable left-Leibniz algebra (see Theorem \ref{V1solvable}). Motivated by the work in \cite{Mason2014}, our analysis continues with a deeper examination of the algebraic structure of the $\mathbb{N}$-graded vertex algebra $V$ when $V_0$ is Gorenstein. By leveraging the Poincar\'{e} duality property of $V_0$, we investigate a particular type of invariant bilinear form on $V_1$, leading us to identify conditions on $V_1$ that generate a rank-one Heisenberg vertex operator algebra, which is embedded within $V$ as a vertex subalgebra. We have included supplementary material in Section \ref{appendix} for readers unfamiliar with Gorenstein algebras.

Let $V=\bigoplus_{n=0}^{\infty}V_n$ be an $\mathbb{N}$-graded vertex algebra such that $\dim V_0<\infty$ and $V_0$ is a local algebra with a maximal ideal $\mathfrak{m}$. Hence, $\mathfrak{m}$ is a Jacobson radical of $V_0$. By Proposition \ref{annihilatorofV_0}, the socle of $V_0$, denoted $soc(V_0)$, is the annihilator of $\mathfrak{m}$. Now, we assume that $V_0=\bigoplus_{d=0}^sV_0^d$ is a graded Gorenstein algebra. Hence, $soc(V_0)$ is a simple $V_0$-module, and $soc(V_0)\cong V_0/\mathfrak{m}$. Moreover, by Proposition \ref{poincareduality}, $V_0$ is a Poincar\'{e}
 duality algebra, and $V_0^s=soc(V_0)=\mathbb{C}t=Ann_{V_0}(\mathfrak{m})$ for some $t\in\mathfrak{m}$. Note that $V_0^0=\mathbb{C}{\bf 1}$ and $\mathbb{C}t$ is the unique minimal ideal of $V_0$. For simplicity, we will use $B(~,~)$ to denote the nondegenerate symmetric invariant bilinear form associated with the Poincar\'{e} duality of $V_0$. So, we have that $B({\bf 1},t)\neq 0$.

\begin{thm}\label{notGorenstein}
    Let $V=\bigoplus_{n=0}^{\infty}V_n$ be an $\mathbb{N}$-graded vertex algebra such that $Ker(D|_{V_0})=\mathbb{C}{\bf 1}$. Assume that $V_1$ contains a simple Lie algebra $\mathfrak{s}$ such that 
    \begin{enumerate}
        \item $\mathfrak{s}+D(V_0)$ is a vertex $V_0$-algebroid and $Leib(\mathfrak{s}+D(V_0))=D(V_0)$.
        \item There exist $u,v\in \mathfrak{s}$ such that $u_1v\neq 0$.
    \end{enumerate} Then $V_0$ is not Gorenstein.
\end{thm}
\begin{proof} By following the proof in Theorem \ref{Notsemisimple}, we can conclude that the bilinear map $(~,~):\mathfrak{s}\times\mathfrak{s}\rightarrow\mathbb{C}$ defined by $(b,b'){\bf 1}=b_1b$ is non-degenerate and a non-zero scalar multiple of the Killing form. Notice that $\mathfrak{s}+ D(V_0)$ is a vertex $V_0$-algebroid that is a semisimple Leibniz algebra. By Proposition \ref{semisimpleLeibniz}, there exists a positive integer $n$ such that $V_0\cong \mathbb{C}[x_1,...,x_n]/(x_ix_j~|~1\leq i,j\leq n)$. Therefore, $V_0$ is not a Gorenstein algebra. \end{proof}
\begin{lem}\label{propmt} Let $V=\bigoplus_{n=0}^{\infty}V_n$ be an $\mathbb{N}$-graded quasi vertex operator algebra such that $V_0=\bigoplus_{d=0}^sV_0^d$ is a graded Gorenstein algebra with the unique maximal ideal $\mathfrak{m}$ and $\dim V_0<\infty$. Now, we set $\mathfrak{a}=Span\{v_0a~|~v\in V_1,a\in \mathfrak{m}\}$. Then the following statements hold
\begin{enumerate}
\item $\mathfrak{a}$ an ideal of $V_0$; 
\item For $a\in \mathfrak{m}$, $v\in V_1$, $v_0a\in\mathfrak{m}$ if and only if $L(1)(a_{-1}v)\in\mathfrak{m}$. In addition, $v_0t\in\mathbb{C}t$ if and only if $L(1)(t_{-1}v)\in\mathbb{C}t$.
\item If $\mathfrak{a}\neq V_0$, then $v_0t\in\mathbb{C}t$ for all $v\in V_1$.
\end{enumerate}
\end{lem}
\begin{proof} Let $\alpha\in V_0$. We have $\alpha_{-1}(v_0a)=-\alpha_{-1}a_0v=-a_0(\alpha_{-1}v)\in\mathfrak{a}$. Hence, $\mathfrak{a}$ is an ideal of $V_0$. This proves (1). 

Now, we recall that for $a'\in V_0$, $b\in V_1$, $L(1)(a'_{-1}b)=a_{-1}L(1)b+b_0a'$. Hence, for $a\in \mathfrak{m}$, $v\in V_1$, $L(1)(a_{-1}v)\in \mathfrak{m}$ if and only if  $v_0a\in\mathfrak{m}$. 
Because $V_0=\mathbb{C}{\bf 1}\oplus\mathfrak{m}$, we then have that  $t_{-1}L(1)v\in\mathbb{C}t$. In addition, $v_0t\in\mathbb{C}t$ if and only if $L(1)(t_{-1}v)\in\mathbb{C}t$. This proves (2). 

Since $\mathfrak{a}\subseteq \mathfrak{m}$, we have that for $a\in\mathfrak{m}$, $v\in V_1$, we have $a_{-1}v_0t=v_0a_{-1}t+(a_0v)_{-1}t=0$. Therefore, $v_0t$ annihilate $a$ for all $a\in\mathfrak{m}$, and $v_0t\in Ann_{V_0}(\mathfrak{m})$. Hence, $v_0t\in\mathbb{C}t$. \end{proof}

\begin{thm}\label{V1solvable}
Let $V=\bigoplus_{n=0}^{\infty}V_n$ be an $\mathbb{N}$-graded vertex algebra such that $V_0$ is a Gorenstein algebra with the unique maximal ideal $\mathfrak{m}$ and $Ker(D|_{V_0})=\mathbb{C}{\bf 1}$. If every subalgebra of the left-Leibniz algebra $V_1$ is closed under $\mathfrak{m}$, and $Leib(V_1)=D(V_0)$, then $V_1$ is a solvable Leibniz algebra.
\end{thm}
\begin{proof} 
    Using the fact that $V_1$ is a left Leibniz algebra and by Proposition \ref{Levi}, $V_1$ is a semidirect product of a semisimple Lie algebra $S$ and $rad(V_1)$, the solvable radical of $V_1$. Assume that $S\neq 0$. Then there is a simple Lie algebra $\mathfrak{s}$ that is contained in $S$ and $\mathfrak{s}$ is isomorphic to $sl_2$. For $s^1,s^2\in \mathfrak{s}$, $a',a''\in V_0$, we have 
$$(s^1+D(a'))_0(s^2+D(a''))=s^1_0s^2+s^1_0 D(a'')=s^1_0s^2+D(s^1_0a'')\in\mathfrak{s}+D(V_0).$$ Hence, $\mathfrak{s}+D(V_0)$ is a subalgebra of the left Leibniz algebra $V_1$. Since $D(V_0)=Leib(V_1)$, we can conclude that  $\mathfrak{s}+D(V_0)$ is a semisimple Leibniz algebra. Since $\mathfrak{m}\cdot \mathfrak{s}+D(V_0)\subset \mathfrak{s}+D(V_0)$, we then have that $\mathfrak{s}+D(V_0)$ is a vertex $V_0$-algebroid. By Proposition \ref{semisimpleLeibniz}, we can conclude that there exists a positive integer $n$ such that $\mathbb{C}[x_1,...,x_n,y_1,...,y_n]/(x_ix_j,y_iy_j,x_iy_j~|~1\leq i,j\leq n)$. This is a contradiction since $V_0$ is a Gorenstein ring. Hence, $S=\{0\}$ and $V_1$ is a solvable Leibniz algebra as desired.
 \end{proof}


Now, we let $V$ be an $\mathbb{N}$-graded quasi vertex operator algebra such that $V_0=\bigoplus_{d=0}^sV_0^d$ is a graded Gorenstein algebra. Let $B(~,~)$ be a non-degenerate invariant bilinear form associated with the Poincar\'{e} duality property of $V_0$. Let $\langle\cdot |\cdot\rangle:V\times V\rightarrow\mathbb{C}$ be a symmetric invariant  bilinear map on $V$ defined by 
\begin{eqnarray*}
    \langle u|v\rangle&=&Res_x x^{-1}B({\bf 1}, Y(e^{xL(1)}(-x^{-2})^{L(0)}u,x^{-1})v).
\end{eqnarray*} Then $\langle~|~\rangle$ satisfies the following properties: for homogeneous elements $a,u,v\in V$,
\begin{eqnarray*}
\langle u|v\rangle &=&B({\bf 1},(-1)^{{wt}(u)}\sum_{i=0}^{\infty}\frac{1}{i!}(L(1)^iu)_{-1-i+2 wt(u)}v),\\
    \langle a_mu|v\rangle&=&\langle u|(-1)^{wt(a)}\sum_{i=0}^{\infty}(\frac{L(1)^i}{i!}a)_{2 wt(a)-m-i-2}v\rangle,\\
    \langle L(-1)u|v\rangle&=&\langle u|L(1)v\rangle,\text{ and }\\
    \langle L(0)u|v\rangle&=&\langle u|L(0)v\rangle.
\end{eqnarray*}

Notice that for $b\in V_1$, we have $0=\langle L(-1){\bf 1}|b\rangle=\langle {\bf 1}|L(1)b\rangle=B({\bf 1},L(1)b )$. Hence, $L(1)V_1\subset Nul(\varepsilon)$. Here, $\varepsilon$ is a linear functional on $V_0$ that is defined by $\varepsilon(a)=B({\bf 1},a)$ for all $a\in V_0$. Since $\dim V_0/Nul(\varepsilon)=1$, it implies that $\dim V_0/L(1)V_1\geq 1$. 

\begin{lem}\label{t-2m}
    Assume that $V_0\neq \mathfrak{a}$. We have $t_{-2}\mathfrak{m}\subseteq V_1^{\perp_{{\langle~,~\rangle}_{V_1}}}$.
\end{lem}
\begin{proof} By Lemma \ref{propmt}, for $u\in V_1$, $a\in\mathfrak{m}$, we have 
  $\langle t_{-2}a|u\rangle=\langle a| t_0u\rangle=B(a,t_0u)=0$.
  Hence, $t_{-2}\mathfrak{m}\subseteq V_1^{\perp_{{\langle~,~\rangle}_{V_1}}}$ as desired.
\end{proof}


Assume that $V_0\neq\mathfrak{a}$. Following \cite{Mason2014} we let $((~,~)):V_1\times V_1\rightarrow \mathbb{C}$ be a bilinear map defined by $((u,v))=B(u_1v,t)$ for $u,v\in V_1$. Clearly, $((~,~))$ is symmetric. Because for $u,v,w\in V_1$,
\begin{eqnarray*}
((u,v_0w))&=&B(u_1(v_0w),t)\\
&=&B(v_0(u_1w),t)-B((v_0u)_1w,t)\\
&=&-B((v_0u)_1w,t)\\
&=&-B((-(u_0v)+D(v_1u))_1w,t)\\
&=&B((u_0v)_1w,t),
\end{eqnarray*} We have $((u_0v,w))=((u,v_0w))$. In addition, $((u_0v,w))=-(v_0u,w))$ for $u,v,w\in V_1$ and $((~,~))$ is an invariant bilinear form on $V_1$.

\begin{lem} Let $V$ be an $\mathbb{N}$-graded quasi vertex operator algebra such that $V_0$ is a graded Gorenstein algebra. Let $B(~,~)$ be a non-degenerate invariant bilinear form associated with the Poincar\'{e} duality property of $V_0$. We define a bilinear map $((~,~)):V_1\times V_1\rightarrow \mathbb{C}$ by $((u,v))=B(u_1v,t)$ for $u,v\in V_1$. Also, we define $rad ((~,~))=\{u\in V_1~|~((u,v))=0\text{ for all }v\in V_1\}$. Assume that $V_0\neq \mathfrak{a}$, then $L(-1)(V_0)\subseteq rad((~,~))$ and $rad((~,~))$ is a 2-sided ideal of the left Leibniz algebra $V_1$.
\end{lem}
\begin{proof} Assume that $V_0\neq \mathfrak{a}$. Hence, $\mathfrak{a}\subseteq \mathfrak{m}$. First, we will show that $L(-1)(V_0)\subseteq rad((~,~))$. Because $t\in Ann_{V_0}\mathfrak{m}$, we can conclude that for $a\in\mathfrak{m}$, $u\in V_1$, we have
$((u,L(-1)a)=B((L(-1)a)_1u,t)=B(u_0a,t)=B(u_0a,t_{-1}{\bf 1})=B(t_{-1}(u_0a),{\bf 1})=B(0,{\bf 1})=0$. Hence, $L(-1)(V_0)\subseteq rad((~,~))$.

Next, we will show that $rad((~,~))$ is a 2-sided ideal of $V_1$. Let $u\in rad((~,~))$ and let $v\in V_1$. Since $((u_0v,w))=((u,v_0w))=0$ for all $w\in V_1$, we can conclude that $u_0v\in rad((~,~))$. Since $u_0v=-v_0u+L(-1)u_1v$ and $L(-1)u_1v\in rad((~,~))$, we can conclude that $rad((~,~))$ is a 2-sided ideal of the left Leibniz algebra $V_1$.
\end{proof}
\begin{prop}\label{rad(())property}
    Let $V$ be an $\mathbb{N}$-graded quasi vertex operator algebra such that $V_0$ is a graded Gorenstein algebra. Assume that $\mathfrak{a}\neq V_0$. The following statements hold:
    \begin{enumerate}
        \item for $a\in\mathfrak{m}$, $u\in V_1$, $u_{-1}a\in rad((~,~))$;
        \item for $v\in rad((~,~))$, $u\in V_1$, we have $v_1u\in\mathfrak{m}$;
        \item for $u\in V_1$, $u\in rad ((~,~))$ if and only if $u_{-1}t\in V_1^{\perp_{\langle~,~\rangle_{V_1}}}$.
    \end{enumerate}
\end{prop}
\begin{proof}
 Recall that for $u,v\in V_1$, $a\in\mathfrak{m}$ we have $v_1u_{-1}a=u_{-1}v_1a-\sum_{i=0}^1\binom{1}{i}(v_iu)_{-i}a=-(v_0u)_0a-(v_1u)_{-1}a$. This implies that for $u,v\in V_1$, $a\in\mathfrak{m}$, $((v,u_{-1}a))=B(v_1u_{-1}a,t)=B(-(v_0u)_0a-(v_1u)_{-1}a,t)=0$.
Therefore, $u_{-1}a\in rad((~,~))$ for all $a\in \mathfrak{m}$, $u\in V_1$. This proves (1).  

Next, we will prove (2). Let $v\in rad((~,~))$. Then $((u,v))=0$ for all $u\in V_1$. Because 
$0=((v,u))=B(v_1u,t)$ for all $u\in V_1$, we can conclude that for $u\in V_1$, $v_1u\in \mathfrak{m}$.

Now, we will prove (3). Observe that for $u,v\in V_1$, $$((u,v))=B(u_1v,t)=\langle v|(-1)\sum_{i=0}^{\infty}\frac{1}{i!}\left(L(1)^iu\right)_{-1-i}t\rangle=-\langle v| u_{-1}t\rangle-\langle v| (L(1)u)_{-2}t\rangle.$$
Because $L(1)u\in V_0=\mathbb{C}{\bf 1}\oplus\mathfrak{m}$, we have that  $(L(1)u)_{-2}t=a_{-2}t=-t_{-2}a+L(-1)t_{-1}a$ for some $a\in\mathfrak{m}$. However, $t_{-1}a=0$. Hence, $(L(1)u)_{-2}t=a_{-2}t=-t_{-2}a$. This implies that $$ 
((u,v))=B(u_1v,t)=-\langle v| u_{-1}t\rangle-\langle v| (L(1)u)_{-2}t\rangle=-\langle v| u_{-1}t\rangle-\langle v| t_{-2}a\rangle.$$ By Lemma \ref{t-2m}, $t_{-2}a\in V_1^{\perp_{\langle~,~\rangle_{V_1}}}$. This  implies that $u\in rad((~,~))$ if and only if $u_{-1}t\in V_1^{\perp_{\langle~,~\rangle_{V_1}}}$.
\end{proof}
\begin{rem}
    The statement (3) in Proposition \ref{rad(())property} is a generalization of Lemma 8 in \cite{Mason2014}.
\end{rem}
\begin{prop} Let $V$ be an $\mathbb{N}$-graded quasi vertex operator algebra such that $V_0$ is a graded Gorenstein algebra. We set $Ann_{V_1}(t_{-1})=\{v\in V_1~|~t_{-1}v=0\}$. Assume that $V_0\neq\mathfrak{a}$. Then 
\begin{enumerate}
\item $Ann_{V_1}(t_{-1})$ is a left ideal of $V_1$.
    \item Let $W=\{w\in V_1~|~\langle t_{-1}w| v\rangle=0\text{ for all }v\in V_1\}$. Then $W$ is a left ideal of $V_1$. 
    \item If $L(1)(V_1)\subseteq \mathfrak{m}$ then $rad((~,~))=W$.
    \end{enumerate}

\end{prop}
\begin{proof} Recall that for $u,v\in V_1$, $a\in V_0$, we have 
$a_{-1}u_0v=u_0a_{-1}v-(u_0a)_{-1}v$. This implies that for $u\in V_1, v\in Ann_{V_1}(t_{-1})$, $t_{-1}u_0v=u_0t_{-1}v-(u_0t)_{-1}v=0.$ Therefore, $u_0v\in Ann_{V_1}(t_{-1})$, and $Ann_{V_1}(t_{-1})$ is a left ideal of $V_1$. This proves the statement (1).

Next, we will prove the statement (2). Clearly, $Ann_{V_1}(t_{-1})\subseteq W$. 
Now, let $w\in W$ and $u\in V_1$. Because $t_0u=\beta t$ for some $\beta\in\mathbb{C}$, we have
\begin{eqnarray*}
    \langle t_{-1}u_0w| v\rangle&=&\langle u_0t_{-1}w-(u_0t)_{-1}w| v\rangle\\
    &=&\langle t_{-1}w|-\sum_{i=0}^{\infty}\frac{1}{i!}(L(1)^iu)_{-i}v\rangle-\langle (u_0t)_{-1}w|v\rangle\\
    &=&\langle t_{-1}w|-(u_0v+(L(1)u)_{-1}v)\rangle-\langle (u_0t)_{-1}w|v\rangle\\
    &=&-\langle (u_0t)_{-1}w|v\rangle\\
     &=&-\beta\langle t_{-1}w|v\rangle\\
     &=&0\text{ for all }v\in V_1.
\end{eqnarray*}
Hence, $W$ is a left ideal of $V_1$. 

Now, we will study statement (3). Let $u\in V_1$. Assume that $t_0u=\mu_u t$ for some $\mu_u\in\mathbb{C}$. Then 
\begin{eqnarray*}
   \langle t_{-1}u|v\rangle&=&\langle u_{-1}t+L(-1)t_0u|v\rangle\\
   &=&\langle u_{-1}t|v\rangle+\mu_{u}\langle L(-1)t|v\rangle\\
   &=&\langle u_{-1}t|v\rangle+\mu_{u}B(t,L(1)(v))\\
   &=&\langle u_{-1}t|v\rangle\text{ for all }v\in V_1.
\end{eqnarray*}
    By the above assumption and Proposition \ref{rad(())property}, we can conclude that $u\in rad((~,~))$ if and only if $u\in W$. \end{proof}

For the rest of this section, we assume that $V$ is an $\mathbb{N}$-graded vertex algebra such that $\dim V_0<\infty$, $\dim V_1<\infty$ and $V_0\neq \mathfrak{a}$. Recall that $\mathbb{C}t$ is a $V_1$-module. We define a linear map
 $\psi:V_1\rightarrow\mathbb{C}t; v\mapsto v_0t$. We define $M=\{u\in V_1~|~u_0t=0\}$ (i.e., $M=Ker(\psi)$). $D(V_0)\subseteq M$ and $M$ has co-dimension at most 1. Observe that for $u,v\in V_1$, we have $(u_0v)_0t=u_0v_0t-v_0u_0t$. If $u\in M$ then $u_0v$ and $v_0u$ are both in $M$. 
\begin{lem}
    Let $V$ be an $\mathbb{N}$-graded quasi vertex operator algebra such that $V_0=\bigoplus_{d=0}^sV_0^d$ is a graded Gorenstein algebra. Assume that $V_0\neq\mathfrak{a}$. Then 
    \begin{enumerate}
        \item $M$ is a solvable ideal of $V_1$ and $M=(L(-1)t)^{\perp_{{\langle~|~\rangle}_{V_1}}}$.
        \item If $u\in M$ then $L(1)u\in\mathfrak{m}$.
        \item If $u\in M$ then $u_{-1}t\in V_1^{\perp_{\langle~,~\rangle_{V_1}}}$ if and only $t_{-1}u\in V_1^{\perp_{\langle~,~\rangle_{V_1}}}$.
        \item At least one of the containments $rad((~,~))\subseteq M$ , $W\subseteq M$ holds.
    \end{enumerate}
\end{lem}
\begin{proof}
    From the above statements, we can conclude immediately that $M$ is an ideal of $V_1$. The solvability property follows from the fact that $V_1$ is solvable.

Observe that for $u\in V_1$, we have
$$ \langle L(-1)t| u\rangle=\langle t_{-2}{\bf 1}| u\rangle=B({\bf 1}, \sum_{i=0}^{\infty}\frac{1}{i!}(L(1)^it)_{-(-2)-i-2}u)=\langle{\bf 1}|t_0u\rangle.$$
 If $u\in M$ then $u\in (L(-1) t)^{\perp_{{\langle~|~\rangle}_{V_1}}}$. Therefore, \begin{equation}\label{Msubset}M\subseteq (L(-1) t)^{\perp_{{\langle~|~\rangle}_{V_1}}}.\end{equation} Using the fact that for $v\in V_1$, $a'\in V_0$, we have $$B(v_0t,a')=\langle v_0t|a'\rangle=-\langle t| v_{0}a'+(L(1)v)_{-1}a'\rangle=-\langle t| (L(1)v)_{-1}a'\rangle.$$
Now, we assume that $v\in(L(-1) t)^{\perp_{{\langle~|~\rangle}_{V_1}}}$.
When $a'\in\mathfrak{m}$, we have $B(v_0t,a')=0$. Similarly, when $a'=\chi{\bf 1}$ for some $\chi\neq 0$, we have 
$B(v_0t,a')=-\chi\langle t| L(1)v\rangle=-\chi\langle L(-1)t| v\rangle=0$. This implies that $B(v_0t,a)=0$ for all $a\in V_0$. Because $B(~,~)$ is non-degenerate, we have $v_0t=0$ and $v\in M$. Hence, \begin{equation}\label{subsetM}(L(-1) t)^{\perp_{{\langle~|~\rangle}_{V_1}}}\subseteq M.\end{equation} By (\ref{subsetM}), (\ref{Msubset}), we have $M=(L(-1)(t))^{\perp_{{\langle~|~\rangle}_{V_1}}}$. This proves (1).

Next, we prove (2) and (3). Let $u\in M$. We then have that $0=\langle u|L(-1)t \rangle=\langle L(1)u|t\rangle=B(L(1)u,t)$. Hence, $L(1)u\in\mathfrak{m}$. Moreover, we have that $\langle u_{-1}t|v\rangle=\langle t_{-1}u-L(-1)t_0u|v\rangle=\langle t_{-1}u|v\rangle$ for all $v\in V_1$. Hence, $u_{-1}t\in V_1^{\perp_{\langle~,~\rangle_{V_1}}}$ if and only $t_{-1}u\in V_1^{\perp_{\langle~,~\rangle_{V_1}}}$.

Next, we prove (4). Suppose that there exists $u\in rad((~,~))$ but $u\not\in M$. Hence, $u_{-1}t\in V_1^{{\perp}_{\langle~,~\rangle_{V_1}} }$ and $u_0t=\mu_u t$ for some non-zero scalar $\mu_u$. We may rescale $u$ so that $\mu_u=1$. Then we have
   $$0=\langle u_{-1}t|v\rangle=\langle t_{-1}u-L(-1)t_0u|v\rangle=\langle t_{-1}u+L(-1)t|v\rangle=\langle t_{-1}u|v\rangle+B(t, L(1)v)$$
for all $v\in V_1$.
    Therefore, $\langle t_{-1}u|v\rangle=-B(t,L(1)v)$ for all $v\in V_1$. Now, we will show that $W\subseteq M$. Let $u'\in W$. Then 
   $ 0= \langle t_{-1}u'|v\rangle=\langle u'|t_{-1}v\rangle$ for all $v\in V_1$. In particular, we have 
    $0=\langle u'|t_{-1}u\rangle=-B(t,L(1)u')=-\langle L(-1)t|u'\rangle$. Hence, $u'\in (L(-1)t)^{\perp_{\langle~,~\rangle_{V_1}}}=M$. 
    \end{proof}

\begin{thm} Let $V$ be an $\mathbb{N}$-graded quasi vertex operator algebra such that $V_0=\bigoplus_{d=0}^sV_0^d$ is a graded Gorenstein algebra, $Ker(L(-1))=\mathbb{C}{\bf 1}$ and $V_0\neq\mathfrak{a}$. Assume that $V_1\neq M$. Hence, there exists $g\in V_1$ such that $g_0t=t$. Moreover, the vector $g$ satisfies the following properties: $g\not\in rad((~,~))$, and $t_{-1}g\in Ann_{V_1}(t_{-1})\cap M$. 

If $g_1g\in \mathbb{C}t\oplus\mathbb{C}^{\times}{\bf 1}$, then $V$ is a module of the rank one Heisenberg vertex operator algebra $M(1)$. In addition, $L(1)g+{\bf 1}\in\mathfrak{m}$. If $V$ is an $\mathbb{N}$-graded indecomposable $M(1)$-module, then there exist $\lambda\in\mathbb{C}$ and a nonnegative integer $k$ such that the graded Gorenstein algebra $V_0$ is isomorphic to $\mathbb{C}[x]/((x-\lambda)^{k+1})$ and one can identify $t$ with $(x-\lambda)^k+((x-\lambda)^{k+1})$.
\end{thm}
\begin{proof} Assume that $V_1\neq M$. So, there exists $g\in V_1$ such that $g_0t\neq 0$ and $g_0t\in\mathbb{C}t$. For simplicity, we may and shall assume that $g_0t=t$. We will show that $g\not\in rad((~,~))$. By direct calculation, we have $(g+\lambda L(-1)t)_0(g+\lambda L(-1)t)=L(-1)(\frac{1}{2}g_1g+\lambda t)$. So, $(g+\lambda L(-1)t)_0(g+\lambda L(-1)t)=0$ if and only if $\frac{1}{2}g_1g+\lambda t\in Ker(L(-1))$. Similarly, it is straightforward to show that 
 $((g+\lambda L(-1)t,g+\lambda L(-1)t))=B(g_1g,t)$.

Using the fact that $Ker(L(-1))=\mathbb{C}{\bf 1}$, we then have the following: \begin{enumerate}
    \item $\frac{1}{2}g_1g+\lambda t\in Ker(L(-1))$ if and only if there exists $\rho_g\in \mathbb{C}$ such that $\frac{1}{2}g_1g+\lambda t=\rho_g{\bf 1}$.
    \item $((g+\lambda L(-1)t,g+\lambda L(-1)t))\in\mathbb{C}^{\times} $ if and only if $\rho_g\neq 0$.
\end{enumerate} 
Let us assume that $\rho_g\in \mathbb{C}^{\times}$. Since $((g,g))=B(g_1g,t)\neq 0$, we can conclude that $g\not\in rad((~,~))$.

Next, we will show that $t_{-1}g\in Ann_{V_1}(t_{-1})\cap M$. Recall that for $a,a'\in V_0$, $u\in V_1$, we have $$a_{-1}(a'_{-1}u)-(a_{-1}a')_{-1}u=(u_0a)_{-1}L(-1)a'+(u_0a')_{-1}L(-1)a.$$ Hence, 
$t_{-1}(t_{-1}g)=(t_{-1}t)_{-1}g+(g_0t)_{-1}L(-1)t+(g_0t)_{-1}L(-1)t=2t_{-1}L(-1)t=L(-1)t_{-1}t=0$. This implies that \begin{equation}\label{t-1g1} t_{-1}g\in Ann_{V_1}(t_{-1}).\end{equation}

Because $g_0t=t$, by Proposition \ref{propmt}, we can conclude that $L(1)t_{-1}g\in\mathbb{C}t$. Since $$0=B(t,L(1)t_{-1}g)=\langle t| L(1)t_{-1}g\rangle=\langle L(-1)t| t_{-1}g\rangle,$$ we can conclude that \begin{equation}\label{t-1g2} t_{-1}g\in (L(-1)t)^{\perp_{\langle~,~\rangle_{V_1}}}=M.\end{equation} In conclusion, $t_{-1}g\in M\cap Ann_{V_1}(t_{-1})$.

Because $g\not\in M=(L(-1)t)^{\perp_{\langle~|~\rangle_{V_1}}}$, we then have that $B(L(1)g,t)=\langle L(1)g|t\rangle=\langle g|L(-1)t\rangle\neq 0.$ Therefore, there exists $q\in\mathbb{C}^{\times}$ such that $L(1)g=q{\bf 1}+u$ for some $u\in \mathfrak{m}$. Let us assume that $g_1g=\rho t+\beta {\bf 1}$. Here, $\rho,\beta\in\mathbb{C}$ and $\beta\neq 0$. We set $h=\frac{1}{\sqrt{\beta}}(g-\frac{1}{2}\rho L(-1)t)$. Clearly, we have that $h_0t=\frac{1}{\sqrt{\beta}}t$. By direct calculation, we have 
    $h_0h=0$ and $h_1h={\bf 1}$. Therefore, $h$ generates a rank one Heisenberg vertex operator algebra $M(1)$, and $V$ is $M(1)$-module.

Using the fact that $L(1)g=q{\bf 1}+u$, we have $L(1)h=\frac{1}{\sqrt{\beta}}L(1)(g-\frac{1}{2}\rho L(-1)t)=\frac{1}{\sqrt{\beta}}(q{\bf 1}+u)$. Since $0=\langle h_0{\bf 1}|t\rangle=-\langle {\bf 1}|h_0t+(L(1)h)_{-1}t\rangle=-\langle {\bf 1}|\frac{1}{\sqrt{\beta}}t+\frac{q}{\sqrt{\beta}}t\rangle=-\left(\frac{1}{\sqrt{\beta}}+\frac{q}{\sqrt{\beta}}\right)B({\bf 1},t),$ and $B({\bf 1},t)\neq 0$ we can conclude that $q=-1$ and $L(1)g=-{\bf 1}+u$. The rest is clear.\end{proof}


\section{Examples}\label{examples}

In this section, we present examples of $\mathfrak{m}$, $t$, $B(~,~)$, $Ann_{V_1}(t_{-1})$, $M$, and $g$ for three types of $\mathbb{N}$-graded vertex algebras $V = \bigoplus_{n=0}^{\infty} V_n$. The vertex algebras in the first example are in a family of rational, simple, $C_2$-cofinite self-dual vertex algebras introduced in \cite{Mason2014}. The second example considers a family of indecomposable, non-simple vertex algebras associated with non-Lie cyclic left-Leibniz algebras, as described in \cite{Barnes2024}. The third example focuses on a family of indecomposable, non-simple vertex algebras related to non-Lie, non-cyclic, nilpotent left-Leibniz algebras. We would like to respectfully acknowledge that the construction of vertex algebroids for this family of vertex algebras was initiated by Barnes, Martin, and Service as part of their FIREbird project during the summer of 2022, under the supervision of the third author of this paper. In this section, we offer a different framework for classifying vertex algebroids in the context of the third example.

\begin{ex}This example appeared in \cite{Mason2014}. We will work out several results that are relevant to this article. Let $\{e,f,h\}$ be Chavalley generators of $sl_2$. We set $H=h/2$. Recall that $L_{\widehat{sl_2}}(k,0)$ is a simple vertex operator algebra corresponding to affine $\hat{sl_2}$ at positive integral level $k$ (WZW model). $L_{\widehat{sl_2}}(k,0)$ has a canonical shift to an $\mathbb{N}$-graded simple vertex operator algebra $L_{\widehat{sl_2}}(k,0)^H$. In addition, 
\begin{itemize}
    \item $V_0$ is spanned by $(e_{-1})^p{\bf 1}$ ($0\leq p\leq k$), and 
    \item $V_1$ is spanned by $h_{-1}(e_{-1})^j{\bf 1}$, $e_{-2}(e_{-1})^i{\bf 1}$ where $0\leq i\leq k-1$, $0\leq j\leq k$. 
    \item In addition, $L_H(1)V_1$ is spanned by $(e_{-1})^p{\bf 1}$ ($0\leq p\leq k-1$), and 
\begin{eqnarray*}
V_0&\cong&\mathbb{C}[x]/(x^{k+1})\\
    V_0/L_H(1)V_1&=&\mathbb{C}(e_{-1})^k{\bf 1}+L_H(1)V_1)\\
    V_1/L_H(-1)V_0&=&Span\{h_{-1}(e_{-1})^j{\bf 1}+L_H(-1)V_0~|~0\leq j\leq k\}.
\end{eqnarray*}
\item $V_0$ is a graded Gorenstein algebra with $t=(e_{-1})^k_{-1}{\bf 1}$, $(e_{-1})^j(e_{-1})^i{\bf 1}=(e_{-1})^{i+j\mod k+1}{\bf 1}$, and the non-degenerate invariant bilinear form $B:V_0\times V_0\rightarrow\mathbb{C}$ defined by 
\begin{eqnarray*}
    B({\bf 1},t)&=&1\\
    B({\bf 1},u)&=&0\text{ when }u\in \bigoplus_{i=0}^{k-1}V_0^i\\
    B(a,a')&=&B({\bf 1},a_{-1}a')\text{ for }a,a'\in V_0.
\end{eqnarray*} Here, $V_0=\bigoplus_{i=0}^kV_0^i$, and $\mathfrak{m}=\oplus_{i=1}^{k}V_0^i$ where $V_0^i=\mathbb{C}e^i_{-1}{\bf 1}$. Also, $Ker L_H(-1)=\mathbb{C}{\bf 1}$. 
\item Since $\dim V_0/L_H(1)V_1$ is one, we have that $V$ is self-dual. 
\item It was shown in Theorem 3 and Corollary 1 of \cite{Mason2014} that 
$$Ann_{V_1}(t_{-1})=rad((~,~))=\text{the nilpotent radical of the left Leibniz algebra $V_1$}.$$ 
\item Also, by the proof of Theorem 6 of \cite{Mason2014}, we have $$M=Span\{h_{-1}(e_{-1})^j{\bf 1},~e_{-2}(e_{-1})^i{\bf 1}~|~0\leq i\leq k-1, 1\leq j\leq k\}.$$
\item $V_1$ is a solvable left Leibniz algebra.
\begin{proof}
Recall that for $b\in V_1$ and $a\in V_0$, $b_{-1}a=a_{-1}b+(b_0a)_{-2}{\bf 1}.$ So, we have
\begin{eqnarray*}
    (h_{-1}a)_0h_{-1}a'&=&\sum_{i=0}^{\infty}(h_{-1-i}a_{i})+a_{-1-i}h_i)h_{-1}a'\\
    &=&h_{-1}a_0h_{-1}a'+a_{-1}h_0h_{-1}a'+a_{-2}h_1h_{-1}a'\\
    &=&h_{-1}(a_0h)_{-1}a'+a_{-1}h_{-1}h_0a'+a_{-2}(h_1h)_{-1}a'\\
    &=&-h_{-1}(h_0a)_{-1}a'+a_{-1}h_{-1}h_0a'+a_{-2}(h_1h)_{-1}a'\text{ for }a,a'\in V_0.
\end{eqnarray*}
Since $h_0$ acts semisimply on $V_0$ and $h_1h=2{\bf 1}$, we have that 
\begin{eqnarray*}
    (h_{-1}a)_0h_{-1}a'&=&-\lambda h_{-1}a_{-1}a'+\mu a_{-1}h_{-1}a'+a_{-2}(h_1h)_{-1}a'\\
    &=&-\lambda h_{-1}a_{-1}a'+\mu (h_{-1}a_{-1}+(a_0h)_{-2})a'+a_{-2}(h_1h)_{-1}a'\\
    &=&-\lambda h_{-1}a_{-1}a'+\mu (h_{-1}a_{-1}-\lambda a_{-2})a'+2a_{-2}a'\\
    &=&(\mu-\lambda)h_{-1}a_{-1}a'+(2-\lambda)a_{-2}a'.
\end{eqnarray*} Here $h_0a=\lambda a$, $h_0a'=\mu a'$. Hence, 
$$(h_{-1}a)_0h_{-1}a'+L_H(-1)V_0=(\mu-\lambda)h_{-1}a_{-1}a'+L_H(-1)V_0.$$ This implies that $V_1/L_H(-1)V_0$ is solvable. Since $L_H(-1)V_0$ is an abelian Lie algebra, we can conclude that $V_1$ is a solvable Leibniz algebra.
\end{proof}
\end{itemize}
\end{ex}
We now present background material that will be highly beneficial for understanding the following two examples.


\begin{prop}\cite{GMS} Let $\mathcal{A}$ be a unital commutative associative algebra and let $\mathcal{B}$ be a vertex $\mathcal{A}$-algebroid. Then one can construct an $\mathbb{N}$-graded vertex algebra $V_{\mathcal{B}}$ such that $(V_{\mathcal{B}})_{(0)}=\mathcal{A}$ and $(V_{\mathcal{B}})_{(1)}=\mathcal{B}$ (under the linear map $v\mapsto v(-1){\bf 1}$) and $V_{\mathcal{B}}$ as a vertex algebra is generated by $\mathcal{A}\oplus \mathcal{B}$. Furthermore, for any $n\geq 1$,
\begin{eqnarray*}
&&(V_{\mathcal{B}})_{(n)}\\
&&=span\{b_1(-n_1).....b_k(-n_k){\bf 1}~|~b_i\in \mathcal{B},n_1\geq...\geq n_k\geq 1, n_1+...+n_k=n\}.
\end{eqnarray*}

\end{prop}
Let $\mathfrak{b}$ denote the 2-dimensional subalgebra of $Vir^+$, spanned by $L_0$ and $L_1$. By a weight $\mathfrak{b}$-module we mean a $\mathfrak{b}$-module on which $L(0)$ acts semisimply.
\begin{dfn}\cite{Li2019} A {\em 1-truncated $sl_2$-module} is a weight $\mathfrak{b}$-module $U=U_0\oplus U_1$ where $L(0)|_{U_0}=0$, $L(0)|_{U_1}=1$, equipped with a linear map $L(-1)\in \Hom(U_0,U_1)$.
\end{dfn} 

Note that for any weight $\mathfrak{b}$-module $U=U_0\oplus U_1$, we have 
$$L(1):U_1\rightarrow U_0,~L(1)|_{U_0}=0.$$
\begin{prop}\label{VBsemiconformal}\cite{Li2019}\ \ 

\begin{enumerate}
    \item 

    Let $U$ be a $\mathfrak{b}$-module. We set $L(U)=U\otimes \mathbb{C}[t,t^{-1}]$. Then $L(U)$ is a $Vir^+$-module with the action given by 
    \begin{eqnarray*}
        &&L(m)(u\otimes t^n)\\
        &=&-(m+n+1)(u\otimes t^{m+n})+(m+1)(L(0)u\otimes t^{m+n})\\
        &&\ \ +\frac{1}{2}m(m+1)(L(1)u\otimes t^{m+n-1})
    \end{eqnarray*}
    for $m\in\mathbb{Z}_{\geq -1}$, $u\in U$, $n\in\mathbb{Z}$.
    \item Let $\mathcal{B}$ be a vertex $\mathcal{A}$-algebroid equipped with a weight $\mathfrak{b}$-module structure on $\mathcal{A}\oplus\mathcal{ B}$ with $L(0)|_{\mathcal{A}}=0$ and $L(0)|_{\mathcal{B}}=1$ such that $L(1)\partial(\mathcal{A})=0$, and 
    \begin{eqnarray}L(1)(u_0v)&=&(L(1)u)_0v+u_0L(1)v\text{ for }u,v\in \mathcal{B},\label{cond1}\\
    L(1)(a\cdot b)&=&a*L(1)b-a_0b\text{ for }a\in \mathcal{A}, b\in \mathcal{B}.\label{cond2}
    \end{eqnarray} Then the $\mathfrak{b}$-module structure on $\mathcal{A}\oplus \mathcal{B}$  can be extended to a $Vir^+$-module structure on $V_{\mathcal{B}}$ which is uniquely determined by 
    \begin{eqnarray*}
        L(-1)&=&D\\
        {[L(m), Y(u,z)]}&=&\sum_{i=0}^2\binom{m+1}{i}z^{m+1-i}Y(L(i-1)u,z)
    \end{eqnarray*} for $u\in \mathcal{A}\oplus \mathcal{B}$, $m\in\mathbb{Z}_{\geq -1}$. Furthermore, $V_{\mathcal{B}}$ is a semiconformal-vertex algebra. 
    \end{enumerate}
\end{prop}


\begin{ex} Let $\mathcal{A}$ be a unital commutative associative algebra with the identity $1_{\mathcal{A}}$. Let $\mathcal{B}$ be a vertex $\mathcal{A}$-algebroid such that $\mathcal{B}$ is a cyclic non-Lie left Leibniz algebra with $\dim \mathcal{B}=2$ and $\mathcal{B}\neq \mathcal{A}\partial(\mathcal{A})$, and $Ker(\partial)=\mathbb{C}1_{\mathcal{A}}$. There exists $b\in \mathcal{B}$ such that $\{b,b_0b\}$ is a basis of $\mathcal{B}$. We set $a=b_1b$. Then $\{b,\partial(a)\}$ is a basis of $\mathcal{B}$ and $\{1_{\mathcal{A}},a\}$ is a basis for $\mathcal{A}$. 

    Assume that $b_0(b_0b)=b_0b$ and $Ker\partial=\mathbb{C}1_{\mathcal{A}}$. If we set $a*a=\alpha_1 1_A+\alpha_2 a$, then we have 
    \begin{enumerate}
        \item $\alpha_1=-\frac{1}{4}\alpha_2^2$, $b_0a=a-\frac{1}{2}\alpha_2 1_\mathcal{A}$,
        \item $a\cdot b=\frac{1}{2}\alpha_2 b+(\frac{1}{2}\alpha_2-1)\partial(a)$,
        \item $a*a=\alpha_2 a-\frac{1}{4}\alpha_2^21_\mathcal{A}$, $a\cdot \partial(a)=\frac{\alpha_2}{2}\partial(a)$,
        \item $\mathcal{A}\cong \mathbb{C}[x]/((x-\frac{1}{2}\alpha_2)^2)$
    \end{enumerate} and $(a-\frac{1}{2}\alpha_2 1_{\mathcal{A}})$ is the unique maximal ideal of $\mathcal{A}$. 
\begin{itemize}    
\item We have $(V_{\mathcal{B}})_0=\mathcal{A}\cong \mathbb{C}[x]/((x-\frac{1}{2}\alpha_2)^2)$ with $\mathfrak{m}=(a-\frac{1}{2}\alpha_2 1_{\mathcal{A}})$ and $t=a-\frac{1}{2}\alpha_2 1_{\mathcal{A}}$. Next, we let $B(~,~)$ be a non-degenerate invariant bilinear form on $(V_{\mathcal{B}})_0$ defined by 
$$B(1_{\mathcal{A}},t)=1,~B(1_{\mathcal{A}},1_{\mathcal{A}})=0,\text{ and }B(t,t)=0.$$ Then $(V_{\mathcal{B}})_0$ is a graded Gorenstein algebra.
\item $V_{\mathcal{B}}$ is a self-dual semiconformal-vertex algebra such that 
\begin{eqnarray*}
    &&L(0)(V_{\mathcal{B}})_0=L(0)\mathcal{A}=0,~ L(0)|_{(V_{\mathcal{B}})_1}=L(0)|_{\mathcal{B}}=1,\\ &&L(1)\partial(V_0)=0\text{ and }L(1)b=1_{\mathcal{A}}.
\end{eqnarray*}

\begin{proof} We assume that $L(0)$ and $L(1)$ satisfy conditions in Proposition \ref{VBsemiconformal} (2) and equations (\ref{cond1}), (\ref{cond2}). Hence, the following statements hold $$L(0)(V_{\mathcal{B}})_0=L(0)\mathcal{A}=0,~L(0)|_{(V_{\mathcal{B}})_1}=L(0)|_{\mathcal{B}}=1,\text{ and }L(1)D(V_0)=0.$$ Now, we will study action of $L(1)$ on $b$. We set $L(1)b=x_11_{\mathcal{A}}+x_2 a$. Since $L(1)a_{-1}b=a_{-1}L(1)b-a_0b$, $a_0b=-a+\frac{1}{2}\alpha_2 1_\mathcal{A}$ and $a_{-1}b=a\cdot b=\frac{1}{2}\alpha_2 b+(\frac{1}{2}\alpha_2-1)\partial(a)$, we have $L(1)a_{-1}b=\frac{1}{2}\alpha_2 L(1)b=\frac{1}{2}\alpha_2 (x_11_\mathcal{A}+x_2 a),$ and
\begin{eqnarray*}
      \frac{1}{2}\alpha_2 (x_11_\mathcal{A}+x_2 a)&=&a_{-1}L(1)b-a_0b\\
      &=&a_{-1}(x_11_\mathcal{A}+x_2 a)+a-\frac{1}{2}\alpha_2 1_\mathcal{A}\\
      &=&x_1a+x_2 a*a+a-\frac{1}{2}\alpha_2 1_\mathcal{A}\\
       &=&x_1a+x_2 (\alpha_2 a-\frac{1}{4}\alpha_2^21_\mathcal{A})+a-\frac{1}{2}\alpha_2 1_\mathcal{A}\\
       &=&(x_1+x_2\alpha_2+1)a+((-\frac{1}{4}\alpha_2^2)x_2-\frac{1}{2}\alpha_2 )1_\mathcal{A}.
  \end{eqnarray*}
  Since $\{1_\mathcal{A},a\}$ is a basis of $\mathcal{A}$, we can conclude that 
  $\frac{1}{2}\alpha_2x_2=x_1+x_2\alpha_2+1$ and $
  \frac{1}{2}\alpha_2x_1=((-\frac{1}{4}\alpha_2^2)x_2-\frac{1}{2}\alpha_2 )$. So, we have $x_1=-\frac{1}{2}\alpha_2x_2-1$ and $L(1)b=(-\frac{1}{2}\alpha_2x_2-1)1_\mathcal{A}+x_2a$. 

  Now, we extend $B(~,~)$ to an invariant bilinear form $\langle~,~\rangle$ on $V_{\mathcal{B}}$. Using the fact that $B(1_\mathcal{A},t)=B(1_\mathcal{A},a-\frac{1}{2}\alpha_2 1_\mathcal{A})\neq 0$, we have $
  B(1_\mathcal{A}, L(1)b)=\langle 1_\mathcal{A},(-\frac{1}{2}\alpha_2x_2-1)1_\mathcal{A}+x_2a\rangle=0$.
  This implies that $0=B(1_\mathcal{A},-1_\mathcal{A}+x_2t) =x_2B( 1_\mathcal{A},t)$. Hence, $x_2= 0$, $L(1)b=-1_\mathcal{A}$ and $\dim (V_{\mathcal{B}})_0/L(1){(V_\mathcal{B})}_1=1$.
\end{proof}
\item If $\alpha_2=0$, we have $rad((~,~))=V_1$. Also, when $\alpha_2\neq 0$, we have $rad((~,~))=\mathbb{C}\partial(t)$.

\begin{proof}
  We will study $rad((~,~))$. By direct calculation, we have 
\begin{eqnarray*}
      &&((b,b))=B(b_1b,t)=B(a,t)=B(a-\frac{1}{2}\alpha_2 1_\mathcal{A}+\frac{1}{2}\alpha_2 1_\mathcal{A},t)=\frac{1}{2}\alpha_2B(1_\mathcal{A},t),\\
      &&((b,\partial(a))=B(b_1\partial(a),t)=B(b_0a,t)=B(a-\frac{1}{2}\alpha_2 1_\mathcal{A},t)=0, \text{ and }\\
&&((\partial(a),\partial(a)))=B(\partial(a)_1\partial(a),t)=0.
      \end{eqnarray*}
Clearly, if $\alpha_2=0$, we have $rad((~,~))=V_1$. Also, when $\alpha_2\neq 0$, we have $rad((~,~))=\mathbb{C}\partial(t)$. 
\end{proof}
\item Assume that $\alpha_2\neq 0$. We have $M=\mathbb{C}\partial(t)$ and $g=b$. When $\alpha_2=2$, $Ann_{{(V_{\mathcal{B}})}_1}(t_{-1})={(V_{\mathcal{B}})}_1$. If $\alpha_2\neq 2$, we have $Ann_{{(V_{\mathcal{B}})}_1}(t_{-1})=\mathbb{C}\partial(t)$. In addition, when $\alpha_2=1$, we then have that $h=b-\frac{1}{2}\partial(a)$ generates $M(1)$, $h_0$ acts semisimply on ${(V_{\mathcal{B}})}_0\oplus {(V_{\mathcal{B}})}_1$. Moreover, $h_0t=t$, $h_0h=0$, and $h_0\partial(t)=\partial(t)$.  

\begin{proof} Since $\partial(a)=\partial(t)\in M$ and $b_0t=b_0a=t$, we can conclude that $M=\mathbb{C}\partial(t)$. In addition, $g=b$ for this case. Observe that 
$$t_{-1}\partial(t)=t_{-1}\partial(a)=a_{-1}\partial(a)-\frac{1}{2}\alpha_2 (1_\mathcal{A})_{-1}\partial(a)=\frac{1}{2}\alpha_2\partial(a)-\frac{1}{2}\alpha_2 \partial(a)=0,$$ and $$t_{-1}b=a_{-1}b-\frac{1}{2}\alpha_2 (1_\mathcal{A})_{-1}b=\frac{1}{2}\alpha_2b+(\frac{1}{2}\alpha_2-1)\partial(t)-\frac{1}{2}\alpha_2 b=(\frac{1}{2}\alpha_2-1)\partial(t).$$ If $\alpha_2=2$ then $Ann_{{V_{\mathcal{B}}}_1}(t_{-1})=V_1$. When $\alpha_2\neq 2$, we have $Ann_{{V_{\mathcal{B}}}_1}(t_{-1})=\mathbb{C}\partial(t)$.

Since $g_1g=b_1b=a=t+\frac{1}{2}\alpha_2 1_\mathcal{A}\in\mathbb{C}t+\mathbb{C}^{\times}1_\mathcal{A}$, we can conclude that $V_{\mathcal{B}}$ is a module of the rank one Heisenberg vertex operator algebra $M(1)$. Now, we will find an element in ${(V_{\mathcal{B}})}_1$ that generate $M(1)$. Observe that 
$(b+y\partial(a))_1(b+y\partial(a))=a+y2b_0a=a+2yt=a-\frac{1}{2}\alpha_2 1_\mathcal{A}+\frac{1}{2}\alpha_2 1_\mathcal{A}+2yt=\frac{1}{2}\alpha_2 1_\mathcal{A}+(2y+1)t$. When we set $y=-\frac{1}{2}$, we have $(b-\frac{1}{2}\partial(a))_1(b-\frac{1}{2}\partial(a))=\frac{1}{2}\alpha_2 1_\mathcal{A}$. Notice that $(b-\frac{1}{2}\alpha_2\partial(a))_0t=t,$ and 
 $(b-\frac{1}{2}\alpha_2\partial(a))_0(b-\frac{1}{2}\alpha_2\partial(a))=b_0b-\frac{1}{2}\alpha_2 b_0\partial(a)=\frac{1}{2}\partial(a)-\frac{1}{2}\alpha_2 \partial(t)=\frac{1}{2}\partial(t)-\frac{1}{2}\alpha_2 \partial(t).$ If $\alpha_2=1$ then $(b-\frac{1}{2}\partial(a))_0(b-\frac{1}{2}\partial(a))=0$. 
 
 In summary, when $\alpha_2=1$, we have $t=a-\frac{1}{2}1_\mathcal{A}$, $L(1)({V_{\mathcal{B}}}_1)=\mathbb{C}1_\mathcal{A}$, $\dim {V_{\mathcal{B}}}_0/L(1){V_{\mathcal{B}}}_1=1$, $M=Ann_{{V_{\mathcal{B}}}_1}(t_{-1})=rad((~,~))=\mathbb{C}\partial(t)$. In addition, $g=b$, $h=b-\frac{1}{2}\partial(a)$ generates $M(1)$, and $h_0$ acts semisimply on $({V_{\mathcal{B}}})_0\oplus ({V_{\mathcal{B}}})_1$. Furthermore, $h_0t=t$, $h_0h=0$, and $h_0\partial(t)=\partial(t)$. 
 \end{proof}
 \end{itemize}
\end{ex}

\begin{ex} Let $\mathcal{B}$ be a non-Lie nilpotent non-cyclic left Leibniz algebra with a basis $\{u,v,w\}$ such that the nonzero left-Leibniz products are given by $u_0v=w$, $v_0w=w$.

\begin{itemize}
\item Let $\mathcal{A}$ be a unital commutative associative algebra with the identity ${\bf 1}_\mathcal{A}$. Then $\mathcal{B}$ is a vertex $\mathcal{A}$-algebroid such that $Ker(\partial)=\mathbb{C}{\bf 1}_\mathcal{A}$. Then the following statements hold: 
\begin{enumerate}
   \item $\mathcal{A}$ is a 2-dimensional commutative associative algebra with a basis $\{{\bf 1}_\mathcal{A}, a\}$ such that $a*a=0$. Here, $a=u_1v$. So, as commutative associative algebras, $\mathcal{A}\cong\mathbb{C}[x]/(x^2)$.
    \item We have $w=\partial(a)$, $u_1u=u_1w=w_1w=0$, $v_1\partial(a)=a$, $v_0a=a$, $u_0a=0$ $a*a=0$, $a\cdot\partial(a)=0$, $a\cdot u=0$. In addition, there exists $\rho\in\mathbb{C}$ such that $v_1v=(1+\rho)1_\mathcal{A}$, $a\cdot v=\rho a$. Hence, $\mathcal{A}\partial(\mathcal{A})=\mathbb{C}\partial(a)$ and $\mathcal{B}/\mathcal{A}\partial(\mathcal{A})$ is an abelian Lie algebra. 
   
\end{enumerate}

\begin{proof}
We set $a=u_1v$. Since $w=u_0v=\partial(u_1v)$ ,
we have $\partial(a)=w$  
and $a\not\in\mathbb{C}{\bf 1}_\mathcal{A}$. We will show that $\{1_\mathcal{A},a\}$ is a basis of $\mathcal{A}$. Let $a'\in \mathcal{A}$. Then $\partial(a')\in V_1$. We set $\partial(a')=\alpha_1u+\alpha_2v+\alpha_3\partial(a)$. Since $u_0v=w, v_0v=0$, and $(\partial(a'))_0v=(\alpha_1 u+\alpha_2v+\alpha_3\partial(a))_0v$, we have 
$\alpha_1w=0$. Therefore $\alpha_1=0$. Similarly, since $u_0w=0$, $v_0w=w$, and $(\partial(a'))_0w=(\alpha_2v+\alpha_3\partial(a))_0w$, we have $0=\alpha_2\partial(a)$. Hence, $\alpha_2=0$ and $\partial(a')=\alpha_3\partial(a)$. This implies that $\partial(a'-\alpha_3 a)=0$. Since $Ker(\partial)=\mathbb{C}1_\mathcal{A}$, we can conclude that $\{1_\mathcal{A},a\}$ is a basis of $\mathcal{A}$.

Next, we will show that $u_0a=0$, $w_1w=u_1u=u_1w=0$. Also, we will calculate $a\cdot u,~a\cdot v,~a\cdot w$, and $v_0a$. Since $\partial(u_0a)=u_0\partial(a)=u_0w=0$, $\partial(v_0a)=v_0\partial(a)=v_0w=w=\partial(a)$, we can conclude that $u_0a=\mu_{u,a}{\bf 1}_\mathcal{A}$ and $v_0a=a+\mu_{v,a}{\bf 1}_\mathcal{A}$ for some $\mu_{u,a},~\mu_{v,a}\in\mathbb{C}$. Because $u_0 u=v_0 v=0$, we have $\partial(u_1 u)=\partial(v_1 v)=0$. Hence, $u_1 u,v_1 v\in\mathbb{C}{\bf 1}_\mathcal{A}$. Also, 
$w_1w=\partial(a)_1\partial(a)=-a_0\partial(a)=0$ and $v_1w=\beta_{v,w}{\bf 1}_\mathcal{A}+a$ for some $\beta_{v,w}\in\mathbb{C}$ because $\partial(a)=w=v_0w=-w_0v+\partial(v_1w)=\partial(v_1w)$. 

Now, we set 
\begin{eqnarray*}
a\cdot u&=&\gamma_uu+\gamma_v v+\gamma_a\partial(a),\\
u_1u&=&\beta_{u,u}{\bf 1}_\mathcal{A},~ v_1v=\beta_{v,v}{\bf 1}_\mathcal{A}.
\end{eqnarray*}
Since $(a\cdot u)_1u=a*(u_1u)-u_0u_0a$, and $u_0a=\mu_{u,a}{\bf 1}_\mathcal{A}$, we have
\begin{eqnarray*}
   a*(\beta_{u,u}{\bf 1}_\mathcal{A})&=&(\gamma_uu+\gamma_v v+\gamma_a\partial(a))_1u\\
   &=&\gamma_u\beta_{u,u}{\bf 1}_\mathcal{A}+\gamma_v a+\gamma_au_0a\\
   &=&\gamma_u\beta_{u,u}{\bf 1}_\mathcal{A}+\gamma_v a+\gamma_a\mu_{u,a}{\bf 1}_\mathcal{A}.
   \end{eqnarray*}
Hence, $(\beta_{u,u}-\gamma_{v})a+(-\gamma_u\beta_{u,u}-\gamma_a\mu_{u,a}){\bf 1}_\mathcal{A}=0$. This implies that 
    $\beta_{u,u}=\gamma_v$, and $
       -\gamma_u\beta_{u,u}=\gamma_a\mu_{u,a}$.
   Since $u_0(a\cdot u)=(u_0a)\cdot u$, we have 
    \begin{eqnarray*}
       u_0(\gamma_uu+\gamma_v v+\gamma_a\partial(a))&=&\mu_{u,a}u\\
       \gamma_v\partial(a)+\gamma_a\partial(u_0a)&=&\mu_{u,a}u\\
       \gamma_v\partial(a)&=&\mu_{u,a}u
   \end{eqnarray*}
Since $\{u,\partial(a)\}$ is linearly independent, we have $\gamma_v=0$ and $\mu_{u,a}=0$. This implies that $\beta_{u,u}=0,~u_1u=0,~a\cdot u=\gamma_uu+\gamma_a\partial(a), ~u_0a=0$. Also, $u_1\partial(a)=u_0a=0$.

Since $v_0(a\cdot u)-a\cdot (v_0u)=(v_0a)\cdot u$, we have 
\begin{eqnarray*}
    v_0(\gamma_uu+\gamma_a\partial(a))&=&(a+\mu_{v,a}{\bf 1}_\mathcal{A})\cdot u\\
  \gamma_a\partial(a)&=&\gamma_u u+\gamma_a\partial(a)+\mu_{v,a}u. \, 
\end{eqnarray*}
Therefore, $\gamma_u=-\mu_{v,a}$, and $a\cdot u=-\mu_{v,a}u+\gamma_a\partial(a)$. Since $(a\cdot u)_1v=a*(u_1v)-u_0v_0a$, we have 
\begin{eqnarray*}
  (-\mu_{v,a}u+\gamma_a\partial(a))_1v&=&a*a-u_0(a+\mu_{v,a}{\bf 1}_\mathcal{A}) \, \,\\
   (-\mu_{v,a}+\gamma_a)a+\gamma_a\mu_{v,a}{\bf 1}_\mathcal{A}&=&a*a .
\end{eqnarray*}
This implies that $a\cdot\partial(a)=\frac{1}{2}\partial(a*a)=\frac{1}{2}(-\mu_{v,a}+\gamma_a)\partial(a)$.

Next, we will show that $a*a=0$, $a\cdot\partial(a)=0$, $v_0a=a$, $v_1\partial(a)=a$, $a\cdot v=\rho\partial(a)$, and $v_1v=(1+\rho)1_\mathcal{A}$ for some $\rho\in\mathbb{C}$. We set $a\cdot v=\rho_u u+\rho_v v+\rho_a\partial(a)$. Since $u_0(a\cdot v)-a\cdot (u_0v)=(u_0a)\cdot v$ we have 
\begin{eqnarray*}
   u_0(\rho_u u+\rho_v v+\rho_a\partial(a))&=&a\cdot \partial(a)\\
   \rho_v\partial(a)&=&\frac{1}{2}(-\mu_{v,a}+\gamma_a)\partial(a).
\end{eqnarray*}
Hence, $\rho_v=\frac{1}{2}(-\mu_{v,a}+\gamma_a)$. 
Since $v_0(a\cdot v)=(v_0a)\cdot v$, we have 
\begin{eqnarray*}
   \rho_av_0\partial(a)&=&(a+\mu_{v,a}{\bf 1}_\mathcal{A})\cdot v \\
   \rho_a \partial(a)&=&\rho_u u+\rho_v v+\rho_a\partial(a)+\mu_{v,a}v.
   \end{eqnarray*}
Hence, $\rho_u=0$, $\rho_v+\mu_{v,a}=0$. Since $\rho_v=\frac{1}{2}(-\mu_{v,a}+\gamma_a)$ we have $\frac{1}{2}(\mu_{v,a}+\gamma_a)=0$. Since $a*a=(-\mu_{v,a}+\gamma_a)a+\gamma_a\mu_{v,a}{\bf 1}_\mathcal{A}$, we have 
\begin{eqnarray*}
    a*a&=&-\mu_{v,a}(2a+\mu_{v,a}{\bf 1}_\mathcal{A}). 
\end{eqnarray*} 
This implies that $a\cdot \partial(a)=-\mu_{v,a}\partial(a)$. In addition, we have $a\cdot v=-\mu_{v,a}v+\rho_a\partial(a)$.

Since $(a\cdot v)_1u=a*(v_1u)-v_0u_0a$, we have 
\begin{eqnarray*}
    (-\mu_{v,a}v+\rho_a\partial(a))_1u&=&a*a\\
    -\mu_{v,a}v_1u&=&a*a\\
    -\mu_{v,a}a&=&-2\mu_{v,a}a-\mu_{v,a}^2{\bf 1}_\mathcal{A}.
\end{eqnarray*}
Since $\{1_\mathcal{A},a\}$ is a basis of $\mathcal{A}$, we can conclude that $\mu_{v,a}=0$. Therefore, $a*a=0$, $a\cdot\partial(a)=0$, $a\cdot v=\rho_a\partial(a)$, $v_0a=a$, and $v_1\partial(a)=a$. In addition, since $a\cdot u=-\mu_{v,a}u+\gamma_a \partial(a)$, $\gamma_a=-\mu_{v,a}$, we can conclude that $a\cdot u=0$.

Since $(a\cdot v)_1v=a*(v_1v)-v_0v_0a$, we have 
\begin{eqnarray*}
    (\rho_a\partial(a))_1v&=&a*(\beta_{v,v}1_\mathcal{A})-v_0(a)\\
    \rho_a a&=&\beta_{v,v}a-a\\
    (\beta_{v,v}-\rho_a-1)a&=&0.
\end{eqnarray*} Hence, $\beta_{v,v}=1+\rho_a$ and $v_1v=(1+\rho_a)1_\mathcal{A}$.
\end{proof}
\item ${(V_{\mathcal{B}})}_{0}=\mathcal{A}$ is a graded Gorenstein algebra with a non-degenerate invariant bilinear form $B(~,~)$ defined by $B(1_\mathcal{A},a)=1$ and $B(1_\mathcal{A},1_\mathcal{A})=0$. $\mathcal{A}$ is a local algebra with the maximal ideal $\mathfrak{m}=(a)$, $soc(\mathcal{A})=\mathbb{C}a$ and $t=a$. Moreover, with the assumption that $v_1v=(1+\rho_a)1_\mathcal{A}$, the following statements hold:
\begin{enumerate}
    \item if $(1+\rho_a)=0$, $M=Ann_{{(V_{\mathcal{B}})}_1}(t_{-1})=Span\{u,\partial(a)\}$ but $rad((~,~))={(V_{\mathcal{B}})}_1$. 
    \item However,  when $(1+\rho_a)\neq 0$, $M=Ann_{{(V_{\mathcal{B}})}_1}(t_{-1})=rad((~,~))=Span\{u,\partial(a)\}$.
\end{enumerate} 
\begin{proof} From the above, $\mathcal{A}$ is a local algebra with the maximal ideal $\mathfrak{m}=(a)$, $Soc(\mathcal{A})=\mathbb{C}a$ and $t=a$. 
Now, we will show that $\mathcal{A}$ is a graded Gorenstein algebra. Let $B:\mathcal{A}\times \mathcal{A}\rightarrow \mathbb{C}$ be an invariant bilinear form defined by $B(1_\mathcal{A},a)=1$ and $B(1_\mathcal{A},1_\mathcal{A})=0$. This implies that $B(a,a)=0$. Hence, $B(~,~)$ is non-degenerate and $\mathcal{A}=\mathbb{C}{\bf 1}\oplus \mathbb{C}a$ is a Poincar\'{e} duality algebra.

Clearly, $Ann_{{(V_{\mathcal{B}})}_1}(t_{-1})=Span\{\partial(a), u\}=M$.
Now, let $((~,~)):\mathcal{B}\times \mathcal{B}\rightarrow \mathbb{C}$ be a bilinear map defined by $((b,b'))=B(b_1b,t)$. We will study $rad((~,~))$. Let $\tau=\alpha_u u+\alpha_v v+\alpha_a\partial(a)\in \mathcal{B}$. Recall that $u_1u=0$, $u_1v=a$, $v_1v=(1+\rho_a)1_\mathcal{A}$ and $v_1\partial(a)=a$, $u_1\partial(a)=0$. We have 
\begin{eqnarray*}
&&((\tau,u))=B((\alpha_u u+\alpha_v v+\alpha_a\partial(a))_1u,t)=B(\alpha_v a,t)=0,\\ &&((\tau,\partial(a)))=B((\alpha_u u+\alpha_v v+\alpha_a\partial(a))_1\partial(a),t)=B(\alpha_v a,t)=0.
\end{eqnarray*}
Observe that \begin{eqnarray*}
    ((\tau,v))&=&B((\alpha_u u+\alpha_v v+\alpha_a\partial(a))_1v,t)\\
    &=&B(\alpha_u a+\alpha_v (1+\rho_a)1_\mathcal{A}+\alpha_a a,t)\\
     &=&B(\alpha_v (1+\rho_a)1_\mathcal{A},t)\\
     &=&\alpha_v (1+\rho_a)B(1_\mathcal{A},t).
\end{eqnarray*}
If $(1+\rho_a)=0$ then $rad((~,~))={(V_{\mathcal{B}})}_1$. Now, let us assume that $(1+\rho_a)\neq 0$. For this case we have $\tau\in rad((~,~))$ if and only if $\alpha_v=0$. In addition, $rad((~,~))=Span\{u,\partial(a)\}$. In summary, when $(1+\rho_a)=0$, $M=Ann_{{(V_{\mathcal{B}})}_1}(t_{-1})=Span\{u,\partial(a)\}$ but $rad((~,~))={(V_{\mathcal{B}})}_1$. However,  when $(1+\rho_a)\neq 0$, $M=Ann_{{(V_{\mathcal{B}})}_1}(t_{-1})=rad((~,~))=Span\{u,\partial(a)\}$.
\end{proof}
\item Assume that $L(0)$ and $L(1)$ on $\mathcal{A}\oplus\mathcal{B}$ satisfy conditions in Proposition \ref{VBsemiconformal} (2) and equations (\ref{cond1}), (\ref{cond2}). Hence, the following statements hold: $L(0)(V_{\mathcal{B}})_0=L(0)\mathcal{A}=0$, $L(0)|_{(V_{\mathcal{B}})_1}=L(0)|_{\mathcal{B}}=1$, and $L(1)D(V_0)=0$. Then $V_{\mathcal{B}}$ is a self-dual semiconformal-vertex algebra with $L(1)u=0$, $L(1)v=-1_\mathcal{A}+ya$ for some $y\in\mathbb{C}$. If $1+\rho_a\neq 0$ or $\rho_a=0$ then $g=v$ generates the rank one Heisenberg vertex operator algebra.

\begin{proof}
Now, we will determine when $V_{\mathcal{B}}$ is a semiconformal -vertex algebra. Assume that $L(1)(V_{\mathcal{B}})_1\neq 0$. We set $L(1)u=x_11_\mathcal{A}+x_2a$, $L(1)v=y_11_\mathcal{A}+y_2 a$, and $L(1)\partial(a)=0$. Since $L(1)a_{-1}u=a_{-1}L(1)u-a_0u$, $a_{-1}u=0$, $a_0u=0$ and $a_{-1}a=0$, we have $0=a_{-1}(x_1 1_\mathcal{A}+x_2 a)=x_1a$. Hence, $x_1=0$ and $L(1)u=x_2a$. Similarly, since $a_{-1}v=\rho_a\partial(a)$, $v_0a=a$, $L(1)a_{-1}v=a_{-1}L(1)v-a_0v$, we have $0=a_{-1}(y_1 1_\mathcal{A}+y_2 a)+a=y_1 a+ a$. Hence, $y_1=-1$ and $L(1)v=-1_\mathcal{A}+y_2a$. Since $L(1)(u_0v)=(L(1)u)_0v+u_0L(1)v$, $u_0v=\partial(a)$, we have $0=(x_2 a)_0v+u_0(-1_\mathcal{A}+y_2a)=-x_2a$. Hence, $x_2=0$ and $L(1)u=0$. There, $L(1)(V_{\mathcal{B}})_1=Span\{1_\mathcal{A}-y_2a\}$. Moreover, $V_{\mathcal{B}}$ is semiconformal, $\dim (V_{\mathcal{B}})_0/L(1)(V_{\mathcal{B}})_1=1 $, and $V_{\mathcal{B}}$ has a unique non-degenerate bilinear form.

Clearly, if either $1+\rho_a\neq 0$ or $\rho_a=0$ then $v$ generates rank one Heisenberg vertex operator algebra $M(1)$. We set $g=a_v v++a_u u+a_a\partial(a)$. Notice $g_0t=a_v v_0t+a_u u_0t+a_a\partial(a)_0t=a_v t$. Hence, $g_0t=t$ if and only if $a_v=1$. Now, we set $a_v=1$, and assume that $g_0g=0$. Because $g_0g=(v+a_u u+a_a\partial(a))_0(v++a_u u+a_a\partial(a))=a_a\partial(a)+a_u\partial(a))$, we can conclude that $a_a=-a_u=0$, and $g=v+a_u(u-\partial(a))$. Next, we will study $g_1g$. By direct calculation, we have
\begin{eqnarray*}
    g_1g&=&(v+a_u(u-\partial(a)))_1(v+a_u(u-\partial(a)))\\
    &=&v_1v+2a_u(a+a)\\
    &=&(1+\rho_a)1_\mathcal{A}+4a_ua\in\mathbb{C}^{\times }1_\mathcal{A}\oplus\mathbb{C}t.
\end{eqnarray*}
If $a_u=0$ then $g=v$, and $g$ generates a rank one Heisenberg vertex operator algebra.
\end{proof}
\end{itemize}
\end{ex}


\section{Appendix}\label{appendix}
\subsection{Nilradical and Jacobson radical} 

Materials in this subsection are based on \cite{Brookes2019}.

Throughout this subsection, $R$ is a commutative ring with an identity $\bf 1$.

\begin{prop} The set of nilpotent elements of a commutative ring $R$ form an ideal ${N}(R)$.
\end{prop}
\begin{dfn}

\ \ 

\begin{enumerate}
    \item (nilradical)
    The ideal ${N}(R)$ is the nilradical of $R$.
    \item (Jacobson radical) The {\em Jacobson radical} of $R$ is the intersection of all the maximal ideals of $R$, written $J(R)$.
\end{enumerate}
\end{dfn}
\begin{rem}\cite{Hungerford1974} Let $R$ be a ring. Then $J(R)$ is the intersection of all the left annihilators of simple left $R$-modules.   
\end{rem}
\begin{prop} Let $k$ be a field of characteristic 0.  

\begin{enumerate}\item Let $R$ be a commutative ring with an identity ${\bf 1}$, we have ${N}(R)\subseteq J(R)$.
\item (weak Nullstellensatz) Let $T$ be a finitely generated $k$-algebra and let $Q$ be a maximal ideal. Then $T/Q$ is a finite field extension of $k$. In particular, if $k$ is algebraically closed and $T=k[X_1,...,X_n]$, the polynomial algebra, then $Q$ is of the form $(X_1-a_1,..., X_n-a_n)$ for some $(a_1,...,a_n)\in k^n$.
\item Let $R$ be a finitely generated $k$-algebra. Then $N(R)=J(R)$.
\item If $R$ is Noetherian then $N(R)$ is nilpotent, i.e. $(N(R))^m=0$ for some $m$.
   \end{enumerate} 
\end{prop}
\begin{prop}
    Let $R$ be a ring that is not necessarily commutative. If $R$ is right Artinian, then $J(R)$ is nilpotent and $R$ is right Noetherian.
\end{prop} 
\begin{dfn}(socle)
    The socle of a non-zero Artinian module $M$, denoted $soc(M)$, is the sum of all the simple submodules of $M$
\end{dfn}
\begin{rem}
    Since $M$ is non-zero Artinian, it does have minimal non-zero submodules, which are necessarily simple. Then $soc(M)\neq 0$.
\end{rem}
\begin{prop}\label{annihilatorofV_0} Let $R$ be a ring and let $M$ be a non-zero Artinian module of $R$. Then $soc(M)=\{m\in M~|~mu=0\text{ for all }u\in J(R)\}$.
\end{prop}
\subsection{Gorenstein Algebras}

Let $\mathcal{A}$ be a commutative associative algebra with an identity ${\bf 1}$ over a field $k$ of characteristic 0. Assume that $\mathcal{A}$ is finitely generated. Then there exist elements $a_1,...,a_n\in \mathcal{A}$ such that every element in $\mathcal{A}$ can be expressed as a polynomial in $a_1,...,a_n$ with coefficients in $k$. Let $\psi:k[X_1,...,X_n]\rightarrow \mathcal{A}$ be a linear map defined by $\psi(X_i)=a_i$ for all $1\leq i\leq n$ and $\psi(X_1^{t_1}...X_n^{t_n})=a_1^{t_1}....a_n^{t_n}$ for all $t_j\geq 0$. $\psi$ is an algebra homomorphism. Clearly $Ker(\psi)$ is an ideal of $k[X_1,...,X_n]$ as an algebra. By Hilbert Basis Theorem, $Ker(\psi)$ is finitely generated and there exist $f_1,...,f_d\in Ker(\psi)$ such that $Ker(\psi)=(f_1,...,f_d)$. Hence, \begin{equation} \label{gorenstein} \mathcal{A}\cong k[X_1,...,X_n]/(f_1,...,f_d).\end{equation}

The commutative associative algebra $\mathcal{A}$ is local if and only if $\mathcal{A}$ has a unique maximal ideal $\mathfrak{m}$.

\subsubsection{Artinian Gorenstein Rings} Materials in this subsection are based on information in \cite{Eisenbud1995, Kock2003}

\begin{dfn} Let $k$ be a field of characteristic 0. Let $\mathcal{A}$ be a commutative Artinian local ring with maximal ideal $\mathfrak{m}$. The ring $\mathcal{A}$ is {\em Gorenstein} if $soc(\mathcal{A})$ is a simple $\mathcal{A}$-module. 
\end{dfn}
\begin{rem} Assume that $\mathcal{A}$ is a Gorenstein. Then there are no nontrivial submodules in $soc(\mathcal{A})$. Since $\mathcal{A}$ is a local ring, this implies that $soc(\mathcal{A})\cong \mathcal{A}/\mathfrak{m}$. Moreover, $soc(\mathcal{A})$ is contained in every nonzero ideal of $\mathcal{A}$. There is a positive integer $n$ such that $\mathfrak{m}^n=0$ while $\mathfrak{m}^{n-1}\neq 0$, and $soc(\mathcal{A})=\mathfrak{m}^{n-1}$. If $\mathcal{A}$ is finite-dimensional vector space over $k$ then $\mathcal{A}$ can be made into a Frobenius algebra. 
\end{rem}
\begin{ex}\ \ 

\begin{enumerate}
    \item $k[X,Y]/(X^2,Y^2)$ is a Gorenstein ring and its socle is generated by $XY+(X^2,Y^2)$.
    \item $k[X]/(X^n)$ is a Gorenstein ring and its socle is generated by $X^{n-1}+(X^n)$.
\end{enumerate}
\end{ex}

\begin{dfn}\cite{Smith1995} Let $k$ be a field of characteristic 0. A finite-dimensional graded $k$-algebra $\mathcal{A}=\oplus_{d=0}^s \mathcal{A}_d$ is called the Poincar\'{e} duality algebra if $\dim_k \mathcal{A}_s=1$ and the bilinear pairing $\mathcal{A}_d\times \mathcal{A}_{s-d}\rightarrow \mathcal{A}_s\cong k$ is non-degenerate for $d=0,...,[s/2]$.
\end{dfn}
\begin{prop}\label{poincareduality}\cite{Geramita2007} A graded Artinian $k$-algebra $\mathcal{A}=\oplus_{d=0}^s\mathcal{A}_d$ is the Poincar\'{e} duality algebra if and only if $\mathcal{A}$ is Gorenstein. If $\mathcal{A}$ is a Gorenstien algebra then $\mathcal{A}_s=soc(\mathcal{A})$
\end{prop}

\subsection{Leibniz algebras} The materials in this subsection are based on the article \cite{DMS}.
\begin{dfn}\cite{DMS, FM}
\begin{enumerate}
    \item A {\em left Leibniz algebra} $\mathfrak{L}$ is a $\mathbb{C}$-vector space equipped with a bilinear map $[~,~]:\mathfrak{L}\times \mathfrak{L}\rightarrow \mathfrak{L}$ satisfying the Leibniz identity $[a,[b,c]]=[[a,b],c]+[b,[a,c]]$ for all $a,b,c\in\mathfrak{L}$.
    \item Let $(\mathfrak{L},[~,~])$ be a left Leibniz algebra over $\mathbb{C}$.
    \begin{enumerate} 
    \item Let $I$ be a subspace of $\mathfrak{L}$. $I$ is a left (respectively, right) ideal of $\mathfrak{L}$ if $[\mathfrak{L},I]\subseteq I$ (respectively, $[I,\mathfrak{L}]\subseteq I$). $I$ is an ideal of $\mathfrak{L}$ if it is both a left and a right ideal.
    \item The series of ideals $....\subseteq \mathfrak{L}^{(2)}\subseteq \mathfrak{L}^{(1)}\subseteq \mathfrak{L}$ where $\mathfrak{L}^{(1)}=[\mathfrak{L},\mathfrak{L}]$, $\mathfrak{L}^{(i+1)}=[\mathfrak{L}^{(i)},\mathfrak{L}^{(i)}]$ is called the derived series of $\mathfrak{L}$. A left Leibniz algebra $\mathfrak{L}$ is solvable if $\mathfrak{L}^{(m)}=0$ for some integer $m\geq 0$. Any left Leibniz algebra $\mathfrak{L}$ contains a unique maximal solvable ideal $rad(\mathfrak{L})$ called the radical of $\mathfrak{L}$ which contains all solvable ideals. 
    \end{enumerate}
    \end{enumerate}
\end{dfn}

We define $Leib(\mathfrak{L})=Span\{[u,u]~|~u\in\mathfrak{L}\}$. The vector space $Leib(\mathfrak{L})$ is, in fact, a solvable ideal of $\mathfrak{L}$. 
\begin{dfn}\cite{DMS}
\begin{enumerate}
 \item A left Leibniz algebra $\mathfrak{L}$ is simple if $[\mathfrak{L},\mathfrak{L}]\neq Leib(\mathfrak{L})$, and $\{0\}$, $Leib(\mathfrak{L})$, $\mathfrak{L}$ are the only ideals of $\mathfrak{L}$.
    \item A left Leibniz algebra $\mathfrak{L}$ is said to be semisimple if $rad(\mathfrak{L})=Leib(\mathfrak{L})$.
\end{enumerate}    
\end{dfn}
\begin{rem}\cite{DMS} The Leibniz algebra $\mathfrak{L}$ is semisimple if and only if $\mathfrak{L}/Leib(\mathfrak{L})$ is semisimple.
\end{rem}
\begin{prop}\label{Levi}\cite{Ba2, DMS} Let $\mathfrak{L}$ be a left Leibniz algebra. 
\begin{enumerate}
    \item There exists a subalgebra $S$ which is a semisimple Lie algebra of $\mathfrak{L}$ such that $\mathfrak{L}$ is a semidirect product of $S$ and $rad(\mathfrak{L})$. As in the case of a Lie algebra, we call $S$ a Levi subalgebra or a Levi factor of $\mathfrak{L}$. 
    \item If $\mathfrak{L}$ is a semisimple Leibniz algebra then $\mathfrak{L}$ is a semidirect product of $S_1\bigoplus\cdots \bigoplus S_k$ and $Leib(\mathfrak{L})$ where $S_j$ is a simple Lie algebra for all $1\leq j\leq k$. Moreover, $[\mathfrak{L},\mathfrak{L}]=\mathfrak{L}$. 
    \item If $\mathfrak{L}$ is a simple Leibniz algebra then there exists a simple Lie algebra $S$ such that $Leib(\mathfrak{L})$ is an irreducible module over $S$ and $\mathfrak{L}$ is a semidirect product of $S$ and $Leib(\mathfrak{L})$.
\end{enumerate}
    
\end{prop}

\bibliographystyle{alpha}
\bibliography{references}

\end{document}